\documentclass[print]{siamart1116}
\usepackage[utf8x]{inputenc}

\newsiamremark{rmk}{Remark}

\usepackage{amssymb}
\usepackage{mathrsfs}
\usepackage{cancel}

\title{Analysis of The Tailored Coupled-Cluster Method in Quantum Chemistry%
\thanks{Submitted to the editors \today.
	\funding{This work has received funding from the Research Council of Norway (RCN) under CoE Grant No.~262695 (Hylleraas Centre for Quantum Molecular Sciences), from ERC-STG-2014 under grant No.~639508, from the Hungarian National Research, Development and Innovation Office (NKFIH) through Grant Nos. K120569, NN110360 and from the Hungarian Quantum Technology NationalExcellence Program (Project No. 2017-1.2.1-NKP-2017-00001). \"O. L. also acknowledges financial support from the Alexander von Humboldt foundation.}
}}

\author{Fabian M. Faulstich%
	\thanks{Hylleraas Centre for Quantum Molecular Sciences, Department of Chemistry, University of
		Oslo, P.O. Box 1033 Blindern, N-0315 Oslo, Norway
		(\email{f.m.faulstich@kjemi.uio.no}).}
	\and
	Andre Laestadius%
	\footnotemark[2]
	\and
	\"Ors Legeza%
	\thanks{Strongly Correlated Systems "Lend\"ulet" Research Group, Wigner Research Centre for Physics, H-1525, Budapest, Hungary.}
	\and
	Reinhold Schneider%
	\thanks{Modeling, Simulation and Optimization in Science, Department of Mathematics, Technische Universit\"at Berlin, Sekretariat MA 5-3, Stra\ss e des 17. Juni 136, 10623 Berlin,Germany}
	\and
	Simen Kvaal%
	\footnotemark[2]
}

\headers{Analysis of The TCC Method in Quantum Chemistry}
{F. M. Faulstich, A. Laestadius, \"O. Legeza, R. Schneider, S. Kvaal}

\begin{document}
\maketitle

\begin{abstract}
In quantum chemistry, one of the most important challenges is the static correlation problem when solving the electronic Schr\"odinger equation for molecules in the Born--Oppenheimer approximation. In this article, we analyze the tailored coupled-cluster method (TCC), one particular and promising method for treating molecular electronic-structure problems with static correlation.
The TCC method combines the single-reference coupled-cluster (CC) 
approach with an approximate reference calculation in a subspace [complete active space (CAS)] of the considered Hilbert space that covers the static correlation. 
A one-particle spectral gap assumption is introduced, separating the CAS from the remaining Hilbert space. This replaces the nonexisting or nearly nonexisting gap between the highest occupied molecular orbital and the lowest unoccupied molecular orbital usually encountered in standard single-reference quantum chemistry.
The analysis covers, in particular, CC methods tailored by tensor-network states (TNS-TCC methods).
The problem is formulated in a nonlinear functional analysis framework, and, under certain conditions such as the aforementioned gap, local uniqueness and existence are proved using Zarantonello's lemma. From the Aubin--Nitsche-duality method, a quadratic error bound valid for TNS-TCC methods is derived, e.g., for linear-tensor-network TCC schemes using the density matrix renormalization group method.
\end{abstract}




\section{Introduction}
In this article, we present an analysis of the coupled-cluster (CC) method tailored by tensor-network states (TNS) for statically correlated electronic systems in quantum chemistry, thereby providing one of the first mathematically rigorous analyses of a multireference CC (MRCC) method.
\newline
\indent
The CC method is today the most popular wavefunction-based computational method in quantum chemistry~\cite{RevModPhys.79.291}. The CCSD(T) scheme, the CC approach with single, double and perturbative triple excitations~ \cite{raghavachari1989fifth,bartlett1990non}, is referred to as the \textit{gold standard of quantum chemistry}, as it yields computational results within error bars of practical experiments for small- and medium-sized molecules at a reasonable cost~\cite{lee1995achieving}. 
However, a severe disadvantage of conventional CC theory is that it fails dramatically for multireference systems, that is, systems whose wavefunction cannot be well approximated by a single Slater determinant reference function \cite{helgaker2014molecular}.
Such systems are said to be \emph{statically correlated}, opposed to systems that are well approximated by a single Slater determinant, which are said to be \emph{dynamically correlated} only. 
\newline
\indent
Even if most molecules are single-reference systems in their equilibrium configuration, multireference character arises even in the simplest of chemical reactions, e.g., dissociation of $\text{N}_2$. Yet, the static correlation problem is a long-lasting challenge in quantum chemistry.
Many different MRCC approaches have been formulated to deal with the problem of static correlation. 
However, aside from formal difficulties and implementational complications, none of these methods have become a widely applicable tool. 
A review of different MRCC approaches is beyond the scope of this article, and we refer to Lyakh \emph{et al.}~\cite{lyakh2012multireference} for a detailed description of the different benefits and disadvantages.
\newline
\indent
We are here concerned with an MRCC method that is based on the single-reference methodology (also called an externally corrected ansatz): 
The tailored CC (TCC) method extends a precomputed solution for a chosen subsystem of the full system by including further electron correlations via CC theory. 
We refer to the subsystem as the complete active space (CAS) and to the remaining system as the external space.
Given the single-reference CC method's major drawback, this subsystem needs to contain the static correlations. 
Consequently, the TCC method can be seen as a special type of an embedding method.
Mathematically this corresponds to a division of excitation operators in two disjoint sub-algebras \cite{kowalski2018properties}.
Nevertheless, in comparison with other ``genuine'' MRCC schemes, the TCC method suffers from the drawback that it is based on a single-reference theory and therewith introduces a certain bias towards a particular reference determinant.
A possible remedy for this drawback is a large CAS covering the static correlations.
The exponential scaling of the CAS makes an efficient approximation scheme for statically correlated systems indispensable for a TCC implementation of practical significance. 
To that end, the TCC method was recently combined with the density matrix renormalization group (DMRG) method. 
The DMRG method \cite{white1999ab} is a high accuracy tool for statically correlated systems~\cite{chan2002highly}, nonetheless, for dynamically correlated problems it requires high computational resources making a wide-ranging application---at this time---intractable.
Hence, it is the symbiosis of the DMRG and the CC method that creates a high efficiency scheme suitable for multireference systems \cite{veis2016coupled,doi:10.1021/acs.jpclett.6b02912,veis2018full,faulstich2019numerical,antalik2019towards,lang2019towards}.    
Granted that the DMRG-TCCSD method is the major motivation for the following analysis, we highlight that the applicability of this article's results exceeds the DMRG-TCCSD method and, more generally, the TNS-TCC method.
\newline
\indent
This paper is organized as follows.
We start by giving a short mathematical introduction to quantum chemistry. 
In Section~\ref{sec:AppOfSe}, we introduce the TCC method with its major caveat: the CAS choice.
Our main results---presented in Section \ref{Sec:Analysis}---rest on certain assumptions that are connected to the structure of the one-particle basis from which the $N$-electron wavefunctions are constructed. 
Generalizing the concept of a HOMO-LUMO gap (see Section~\ref{subsec:LocalUniqueness}), we introduce a gap between the CAS and the external space (Assumption (A)). 
This allows us to derive various norm equivalences that can be used to establish continuity of the considered cluster operators with respect to different topologies. 
Moreover, a more technical constraint (Assumption (B)) enters our analysis when we assume that the fluctuation potential, i.e., an operator that models a part of the electron--electron interaction, cannot be too large when restricted to the external space.
This manifests the importance of a well-chosen CAS as mentioned above. 
Also, as far as the multireference character of systems included in our treatment is concerned, we have to assume that those determinants that contribute the most in the $N$-electron CAS have energies very similar to the reference determinant. 
Other determinants can contribute too, but their weight must become smaller the larger the energy difference with respect to the reference determinant becomes. 
We then use Zarantonello's lemma to derive local existence and uniqueness of TCC solutions under {Assumption} (A) and (B). 
In Section~\ref{Sec:Error}, we perform an energy error analysis and present major differences to the single-reference CC method. 
Via the Aubin--Nitsche-duality method we are able to derive a quadratic energy error bound valid for TCC schemes like the TNS-TCC method.

\section{The Electronic Schr\"odinger Equation}
\label{sec:SE}
In general, a Hamilton operator is an elliptic differential operator, formally defined by
\begin{equation}
H\psi = -\frac{1}{2}\Delta \psi +V\psi~.
\label{Eq:Hgen}
\end{equation}
The function $V:\mathbb{R}^n\rightarrow \mathbb{R}$ is called the potential of the operator.
Such differential operators are in general well studied \cite{evans2010partial,garcke2000computation,gustafson2011mathematical,renardy2006introduction}.
However, the numerical treatment of physical systems, especially electronic systems, is still challenging. 
In the spirit of mathematical rigor, we summarize the weak formulation of the Hamilton operator in Eq.~\eqref{Eq:Hgen}:

\textit{
The Hamilton operator induces a bilinear form $\mathcal{A}_V: \mathcal{C}^{\infty}_c(\mathbb{R}^n)\times  \mathcal{C}^{\infty}_c(\mathbb{R}^n)$ by
\begin{equation}
\mathcal{A}_V(\tilde \psi,\psi) =\frac{1}{2}\langle\nabla \tilde \psi, \nabla\psi\rangle_{(L^2(\mathbb{R}^{n}))^{n}}+\langle \tilde \psi,  V \psi \rangle_{L^2(\mathbb{R}^{n})}~,
\label{eq:quad}
\end{equation}
where $\mathcal{C}^{\infty}_c(\mathbb{R}^n)$ is the space of smooth functions on $\mathbb{R}^n$ with finite support.
Assuming boundedness of $V(x)(\cdot): \mathcal{C}^{\infty}_c(\mathbb{R}^n)\to L^2(\mathbb{R}^n)$, the Cauchy-Schwarz inequality yields
$\mathcal{A}_V(\tilde \psi,\psi)\leq C\Vert \tilde \psi \Vert_{H^1(\mathbb{R}^{n})} \Vert \psi \Vert_{H^1(\mathbb{R}^{n})}$, for all $\tilde \psi,\psi\in \mathcal{C}^{\infty}_c(\mathbb{R}^n)$.
Since $ \mathcal{C}^{\infty}_c(\mathbb{R}^n)$ is dense in $H^1(\mathbb{R}^{n})$, we can extend $\mathcal{A}_V$ to a bounded and symmetric bilinear form on $H^1(\mathbb{R}^{n})\times H^1(\mathbb{R}^{n})$. }

Subsequently we omit the domain of the function space whenever it is clear from context.
In this article, we assume that $H$ satisfies G\aa rding's inequality \cite{renardy2006introduction}, i.e., there exist $c,e\in\mathbb R$ with $c>0$ such that
\begin{equation}
\mathcal{A}_V(\psi,\psi) +e\langle  \psi,\psi\rangle_{L^2} \geq c \Vert \psi\Vert_{H^1}^2~.
\label{eq:Gard}
\end{equation}
We furthermore define the Rayleigh--Ritz quotient $\mathcal R_V(\psi) = \mathcal{A}_V(\psi,\psi)/\langle\psi,\psi\rangle_{L^2}$ for all $\psi\in H^1\setminus \{0\} $. Then $E_0 = \inf_{\psi\in H^1 \setminus \{0\} } \mathcal R_V(\psi)$ is well defined even though the infimum need not be attained. 
However, if such a minimizer exists it is called a ground state. 
Under the assumption that $H$ attains a ground state $\psi_0\in  H^1$ we can recast the Schr\"odinger equation $\mathcal A_V(\tilde \psi,\psi_0) = E_0 \langle \tilde{\psi},\psi_0\rangle_{L^2}$ for all $\tilde \psi\in H^1$ (i.e. $H \psi_0 = E_0 \psi_0$) by means of the Rayleigh--Ritz variational principle:
\begin{equation}
E_0=\min_{\psi\in H^1 \setminus \{0\} }  \mathcal R_V(\psi) ~.
\label{eq:Rayleigh-Ritz variational principle1}
\end{equation}
Note, whenever $
\gamma = \inf \{\mathcal R_V(\psi): \psi\in H^1 , \psi\neq 0 , \langle \psi_0,\psi\rangle_{L^2} =0 \}-E_0 >0$, $\psi_0$ is (up to a phase) the unique ground state of $H$ and $\gamma$ is called the spectral gap. 
\newline
\indent
This article focuses on the electronic Schr\"odinger equation obtained from the Born--Oppenheimer approximation \cite{slater1964quantum,born1927quantentheorie}. 
In Hartree atomic units, the Hamilton operator of a Coulomb system that consists of $N$ electrons and $N_{\mathrm{nuc}}$ nuclei reads
\begin{align*}
H\psi(x)
&= 
-\sum_{i=1}^{N}
\frac{1}{2} \Delta_i\psi(x)
+
\underbrace{
\Big(
\frac{1}{2}\sum_{i=1}^{N}\sum_{j\neq i}^{N}\frac{1}{|r_i-r_j|}
-
\sum_{i=1}^{N}\sum_{j=1}^{N_{\mathrm{nuc}}}\frac{Z_j}{|r_i-R_j|}\Big)}_{=V_C}\psi(x)~,
\label{Hamiltonian}
\end{align*}
with $V_C$ the Coulomb potential.
Here $\psi(x)= \psi(x_1,\dots,x_N)$, where the argument $x_i=(r_i,s_i)$ for $i\in\{1,...,N\}$ is associated with the position of the $i$-th electron $r_i\in\mathbb{R}^3$ and its spin $s\in\{\pm 1/2\}$. 
As a result of the Born--Oppenheimer approximation, the nuclei positions $R_j\in\mathbb{R}^3$ and charges $Z_j>0$, $j \in \{1,\dots,M\}$, enter as fixed parameters. 
This general formulation is so far independent of spin as an explicit variable. 
Moreover, solutions to the above Hamiltonian do not naturally fulfill Pauli's principle, i.e., fermionic state functions need to be antisymmetric with respect to permutations of the coordinates $x_i$. 
Considering these further constraints, the set of admissible wavefunctions is given by
\begin{equation}
\mathcal{H}
= H^1\left(\left(\mathbb{R}^3\times\left\lbrace\pm\frac{1}{2}\right\rbrace\right)^N\right)\cap\bigwedge_{i=1}^NL^2\left(\mathbb{R}^3\times\left\lbrace\pm\frac{1}{2}\right\rbrace\right)~,
\end{equation}
where $\wedge$ is the antisymmetric tensor product that guarantees Pauli's principle. 
We conclude, the minimization problem Eq.~\ref{eq:Rayleigh-Ritz variational principle1} corresponding to electronic structure calculations is given by
\begin{equation}
\label{eq:Rayleigh-Ritz variational principle}
E_0=\min_{\psi\in \mathcal{H} \setminus \{0\} }  \mathcal R_{V_C}(\psi)~.
\end{equation}
   
\begin{rmk}
\label{Rmk:Gelfand}
The Hamilton operator is here a map $H:H^1\supseteq \mathcal{H}\to H^{-1}$, where $ H^{-1}$ is the dual space of $H^1$. In particular, this means that instead of the $L^2$-inner product we need to consider the dual pairing $\langle\cdot,\cdot\rangle_{H^1,H^{-1}}$.
To justify the use of the inner product we recall that $H^1$ is continuously embedded in $L^2$ and that $H^1$ is dense in $L^2$, i.e, $H^1$ is densely embedded in $L^2$ and we write $H^1\overset{d}{\hookrightarrow} L^2$.
For such a Hilbert space structure, we define the Gelfand triple $H^1\overset{d}{\hookrightarrow} L^2\overset{d}{\hookrightarrow}H^{-1}$ (also called rigged Hilbert space), identifying $L^2\simeq (L^{2})'$.
Note that as a consequence we are no longer allowed to identify $H^1\simeq H^{-1}$.
One advantage of the Gelfand triple is that the use of the $L^2$ inner product instead of the dual pairing $\langle\cdot,\cdot\rangle_{H^1,H^{-1}}$ becomes meaningful~\cite{wlokapartial}:
Given the Gelfand triple $H^1\overset{d}{\hookrightarrow} L^2\overset{d}{\hookrightarrow}H^{-1}$ and the scalar product $\langle\cdot,\cdot\rangle_{L^2}$ on $L^2\times L^2$, we find $\langle x,y\rangle_{L^2} = \langle x,y\rangle_{H^1 \times H^{-1}}$ for all $x\in H^1$ and $y\in L^2$ since $H^1 \subseteq L^2$ and $L^2 \subseteq H^{-1}$.
By Hahn--Banach we can therefore continuously extend $\langle x, \cdot\rangle_{L^2}$ from $L^2$ to $H^{-1}$ for arbitrary but fixed $x\in H^1$.
\end{rmk}

Remark~\ref{Rmk:Gelfand} becomes important when considering quantum molecular systems on the infinite dimensional Hilbert space $\mathcal H$.
We subsequently make use of the inner product notation, emphasizing that the reader should keep this detail in mind. 
Moreover, henceforth we use the short notation $\langle \cdot,\cdot\rangle$ rather than $\langle\cdot,\cdot\rangle_{L^2}$ or $\langle\cdot,\cdot\rangle_{l^2}$ whenever the meaning is clear from context.

\section{Approximate Solutions of the Schr\"odinger Equation}
\label{sec:AppOfSe}

The high dimensionality of Eq.~\eqref{eq:Rayleigh-Ritz variational principle} makes a direct minimization in general intractable.
The variety of possible approximations, depending on the chemical problem and required accuracy, is rich \cite{helgaker2014molecular,martin2004electronic,hjorth2017advanced}.
However, most wavefunction based schemes rely on an antisymmetrized product ansatz.
The factors of this exterior product are called \emph{spin-orbitals} and the functions spanning the solution space are denoted \emph{Slater determinants}. 
Subsequently, we denote the spin-orbitals by $\chi$ and Slater determinants by $\phi$.
For an $N$-electron problem, let $N<K$ and $\mathscr B = \{\chi_1,...,\chi_K\}\subseteq  H^1(\mathbb{R}^3\times \{\pm\frac{1}{2}\})$ denote an $L^2(\mathbb{R}^3\times \{\pm\frac{1}{2}\})$-orthonormal set of functions, called spin-orbitals.
An $N$-particle wavefunction fulfilling Pauli's exclusion principle is obtained by forming the exterior product of $N$ spin-orbitals $\{\chi_{\mu_1},...\chi_{\mu_N}\}$
\begin{equation}
\phi[\mu_1,...,\mu_N](x_1,...,x_N)= \frac{1}{\sqrt{N!}} \bigwedge_{i=1}^N \chi_{\mu_i}(x_1,...,x_N)   =\frac{1}{\sqrt{N!}}\det(\chi_{\mu_i}(x_j))_{i,j=1}^N~,
\label{SD}
\end{equation}
where the indices $\mu_1,...,\mu_N\in\{1,...,K\}$ are in canonical order, i.e., $\mu_1<...<\mu_N$. 
We see immediately that Slater determinants inherit $L^2$-orthonormality from the spin-orbital basis.
The corresponding Galerkin space $\mathcal{H}_K$ is then spanned by all possible exterior products of length $N$ in $\mathscr B$.
This construction yields a combinatorial scaling of $\mathcal{H}_K$---also called the full configuration-interaction (FCI) space.
An $L^2$-orthonormal basis $\mathcal B_K$ of $\mathcal{H}_K$ is obtained by imposing a canonical ordering of the spin-orbitals in the exterior products, i.e.,
\begin{equation*}
\mathcal B_K = \{\phi[\mu_1,...,\mu_N] :  \mu_i\in \{1,..., K \}, \mu_1<...<\mu_N\}~.
\end{equation*}
Subsequently we use the notation $\phi_\mu =\phi[\mu_1,...,\mu_N] $ and without loss of generality define the reference determinant $\phi_0=\phi[1,...,N]$. 
Furthermore, we use the standard terminology of quantum chemistry and call spin-orbitals defining $\phi_0$ occupied and the remaining ones virtual.
Indices $I,J,K,...$ are assumed to be occupied (i.e. smaller or equal than $N$) while $A,B,C,...$ are assumed to be virtual (i.e. greater than $N$).
\newline
\indent
Essential to the CC theory is the $L^2$-bounded commutative algebra of cluster operators $\mathcal{C}_K$, defined via single-excitation operators.
We define a single-excitation operator $X_I^A$ as follows: $X_I^A\phi_\mu$ replaces $\chi_I$ by $\chi_A$ for any $\phi_\mu$ if $\mu_i=I$ for some $i$ and $\mu_j\neq A$ for all $j$, otherwise $X_I^A\phi_\mu = 0$.
Since Slater determinants are normalized, this defines $X_I^A$ as an $L^2$-bounded operator. 
Higher order excitation operators are then defined as product of single-excitation operators, e.g ., the double excitation operator $X_{IJ}^{AB} = X_I^A X_J^B$.
The fermionic commutation relations, i.e., $[a^{\,}_i, a^\dagger_j]_+ = \delta_{i j}$ and  $[a_{i}^{\dagger },a_{j}^{\dagger }]_+=[a_{i}^{\,},a_{j}^{\,}]_+=0$, yield that excitation operators commute.
The set of excitation operators is then trivially an $L^2$-bounded and commutative algebra.
Furthermore we define the rank of an excitation operator as the length of the product, when written as product of single-excitation operators.
Note that by antisymmetry of Slater determinants, the product $X_{I_1... I_n}^{A_1... A_n}$ is antisymmetric under permutations of $\{I_1, ..., I_n\}$ and $\{A_1, ..., A_n\}$, respectively. 
Similar to $\mathcal{H}_K$, a basis of $\mathcal{C}_K$ is obtained by imposing a canonical ordering of the product of single-excitation operators with respect to the orbital indices.

\begin{proposition}
\label{rmk:isometricallt isomorphic}
We can induce a norm on $\mathcal{C}_K$ via $\Vert X_\mu\Vert_{\mathcal{C}_K} = \Vert X_\mu \phi_0 \Vert_{H^1}$.
Then $\mathcal{C}_K$ is isometrically isomorphic to $\text{span}\{\phi_0\}^{\perp}$, where $\perp$ denotes $L^2$-orthogonal complement in $\mathcal{H}_K$.
\end{proposition}
\begin{proof}
For any $\phi_\mu \in \mathcal{B}_K$, there exists a unique excitation operator such that $\phi_\mu = X_\mu \phi_0$ up to a sign factor, i.e., $\phi_\mu$ is generated from $\phi_0$ by repeated substitution of occupied spin-orbitals. 
Conversely, for any excitation operator $X_\mu$ there is a unique $\phi_\mu\in \mathcal{B}_K$ such that $\phi_\mu = X_\mu \phi_0$ up to a sign factor.
Hence, we can define a bijective homomorphism between $\mathcal{C}_K$ and $\text{span}\{\phi_0\}^{\perp}$ , where we impose canonical vector-space operations on the respective spaces, i.e., vector addition and scalar multiplication.
By construction this map is trivially an isometry, which proves the claim.
\end{proof}

Subsequently, we refer to the basis index $\mu$ as an excitation index and switch to the more common multi-index notation, i.e., $\mu=\binom{A_1,...,A_r}{I_1,...,I_r}$ with occupied indices $\{I_1,...,I_r\}$ and virtual indices $\{A_1,...,A_r\}$. 
The set of all possible excitation indices is denoted $\mathcal{J}$, where we dropped the dependence on $K$ and the reference state due to notational simplicity.
Using the canonical ordering, the number of possible excitation indices up to a certain excitation rank $n\leq N$ is given by
\begin{equation*}
|\mathcal{J}|=\sum_{k=1}^n\binom{N}{k}\binom{K-N}{k}~.
\end{equation*}  

In practice the spin-orbitals in $\mathscr{B}$ and thus the reference wavefunction $\phi_0$ come from a preliminary Hartree--Fock calculation \cite{helgaker2014molecular,lieb1977hartree,lions1987solutions}: 
In a nutshell, starting with an initial spin-orbital basis $\{\chi_i^{(0)}\}_{i=1}^K$ we minimize Eq.~\ref{eq:Rayleigh-Ritz variational principle} with a mean-field potential.
This yields a nonlinear $K$-dimensional eigenvalue problem $\bar{F}(\chi_1,\cdots,\chi_N) \chi_i = \lambda_i \chi_i$, for $i = 1,\cdots,N$, where the Fock matrix $\bar{F}$ depends on the $N$ occupied spin-orbitals.
The Fock matrix is symmetric, implying that the $N$ eigenvectors can be completed with $K-N$ additional eigenvectors. It is these eigenfunctions that form $\mathscr{B}$. 
\newline
\indent
We observe that the Hartree--Fock calculation depends on the dimension $K$ in a manner which is not entirely controlled:
In general, it is unclear whether the $\{\chi_i\}_{i=1}^K$ form a global minimum of the Rayleigh--Ritz minimization problem and whether the solution converges as $K\to\infty$. 
Such questions are beyond the scope of the present article, but is relevant in context of the $K\to\infty$ limit of the TCC method, see Remark~\ref{rmk:OnTheStructureOfFockOperator}. 
The Hartree--Fock calculation induces a splitting of the Hamilton operator $H = F + W$ with $F = \sum_{i=1}^N \bar{F}(i)$, where $\bar{F}(i)=I\otimes\ldots\otimes I\otimes\bar{F}_{(i)} \otimes I\otimes...\otimes I$ indicating by $\bar{F}_{(i)}$ that $\bar{F}$ appears on the $i$-th position in the Kronecker product. 
Subsequently, we will refer to $F$ as the Fock operator and to $W$ as the fluctuation potential.
\newline
\indent
We define for any multi-index $\mu$ of excitation rank $n\leq N$ the number 
\[
\varepsilon_\mu = \sum_{j=1}^n (\lambda_{A_j} - \lambda_{I_j})~,
\]
i.e., the sum of the single-particle Hartree--Fock energy differences of the occupied and virtual spin-orbitals in $\mu$.
Defining $\Lambda_0 = \sum_{i=1}^N \lambda_i$---the sum over the $N$ first single-particle Hartree--Fock energies---we see that the Slater determinants $\mathcal{B}_K$, formed by the single-particle Hartree--Fock eigenfunctions, are the $N$-particle Hartree--Fock eigenfunctions with $F \phi_\mu = (\Lambda_0 + \varepsilon_\mu) \phi_\mu$.

Returning to the Schr\"odinger equation, the $L^2$-normalization constraint on $\psi \in\mathcal{H}_K$ is subsequently replaced by the intermediate normalization, i.e., $\langle \phi_0,\psi\rangle=1$.
Hence, $\psi=(I+S)\phi_0$ holds for an operator $S =\sum_{\mu\in\mathcal{J}}s_{\mu}X_{\mu} \in \mathcal{C}_K$ and we denote the basis coefficients $(s_{\mu})_{\mu\in\mathcal{J}}= (\langle \phi_\mu,\psi\rangle )_{\mu\in\mathcal{J}}$ \emph{excitation amplitudes}. 
Inserting this parameterization of wavefunctions into the Schr\"odinger equation, we find that Eq.~\eqref{eq:Rayleigh-Ritz variational principle} is equivalent to the linear problem
\begin{equation}
\left\lbrace
\begin{aligned}
E_0^{(\mathrm{FCI})} &=\langle \phi_{0},H\psi_0^{(\mathrm{FCI})} \rangle~,\\
0&=\langle \phi_{\mu},(H-E_0^{(\mathrm{FCI})})\psi_0^{(\mathrm{FCI})} \rangle,\quad\forall\mu\in\mathcal{J}~,
\end{aligned}\right.
\label{eq:FCI}
\end{equation}
which is known as the FCI scheme. 
For a derivation of the corresponding amplitude equations we refer the reader to \cite{helgaker2014molecular}.
\subsection{Projected Single-Reference Coupled-Cluster Method}
The previously described FCI approach suffers from the \emph{curse of dimensionality} since $\mathcal{H}_K$ grows exponential with the number of particles, i.e., $\dim(\mathcal{H}_K)$ $\in\mathcal{O}(K^N)$.
Furthermore, truncating the operator $S\in \mathcal{C}_K$ at rank-$n$ excitations reduces the computational cost but yields CI methods that are no longer energy size-extensive nor size-consistent \cite{helgaker2014molecular}, which are quantum chemical concepts relating to the correct energy behavior with respect to the system's size and dissociation \cite{shavitt2009mbpt}.
Alternatively to the linear manifold used in Eq.~\eqref{eq:FCI}, an exponential parameterization of wavefunctions can be used \cite{hubbard1957description,hugenholtz1957perturbation}:
Let $\psi\in\mathcal{H}_K$ be intermediately normalized, i.e., $\psi = (I+S)\phi_0$ for some $S\in\mathcal{C}_K$.
Then there exists a unique $T\in\mathcal{B}(H^1,H^{-1})$ with $\psi=e^T\phi_0$ \cite{schneider2009analysis} (for the result in the limit $K\to\infty$ see \cite{rohwedder2013continuous}), where
\begin{equation}
T =\sum_{\mu\in\mathcal{J}}t_{\mu}X_{\mu} \quad\text{and}\quad  T = \log(I +S) ~.
\end{equation}
This exponential parameterization has the benefit that it is multiplicatively separable with respect to subsystems that are separated by distance, thereby regaining size-extensivity and consistency under mild assumptions on the reference determinant~\cite{shavitt2009mbpt}. 

To solve the Schr\"odinger equation, it remains to determine the \emph{cluster amplitudes} $(t_{\mu})_{\mu\in\mathcal{J}}$. 
This is the pursuit of the CC method. The linked CC equations describing the cluster amplitudes are given by \cite{helgaker2014molecular}:
\begin{equation}
\left\lbrace
\begin{aligned}
E^{(\mathrm{CC})}_0 &=\langle\phi_0,e^{-T}He^{T}\phi_0\rangle~,\\
0&=\langle\phi_{\mu},e^{-T}He^{T}\phi_0\rangle~, \quad\forall\mu\in\mathcal{J}~.
\end{aligned}\right.
\label{eq:LinkedCC}
\end{equation}
The equivalence to the Schr\"odinger equation~\eqref{eq:Rayleigh-Ritz variational principle}, is straightforwardly established \cite{helgaker2014molecular}:
Given an intermediately normalized minimizer of Eq.~\eqref{eq:Rayleigh-Ritz variational principle} $\psi = e^T\phi_0$, we obtain
\begin{equation*}
E_0 e^T\phi_0 = H e^T\phi_0 \Rightarrow E_0 \phi_0 = e^{-T} H e^T\phi_0 \Rightarrow \left\lbrace 
\begin{aligned}
E_0 &=\langle\phi_0,e^{-T}He^{T}\phi_0\rangle~,\\
0&=\langle\phi_{\mu},e^{-T}He^{T}\phi_0\rangle~, \quad\forall\mu\in\mathcal{J}.
\end{aligned}\right.
\end{equation*}
Conversely, given a solution $\psi = e^T\phi_0$ fulfilling Eqs.~\eqref{eq:LinkedCC}, we find
\begin{equation*}
\begin{aligned}
H e^T\phi_0 &= e^{T}e^{-T}H e^T\phi_0 = \sum_{\mu \in \mathcal{J}} e^{T} \phi_{\mu}\langle \phi_{\mu} , e^{-T}H e^T\phi_0\rangle + e^{T} \phi_{0}\langle \phi_{0} , e^{-T}H e^T\phi_0\rangle \\
&= E_0^{(CC)} e^{T} \phi_{0}~.
\end{aligned}
\end{equation*}
Note that this equivalence does in general not hold true under truncations of $T$, e.g., considering only single- and double-excitations in $T$ (the CCSD method).
In this case, the CC method is no longer variational. 
For a more detailed discussion on this topic see~\cite{laestadius2019coupled}.

We emphasize that there exists a one-to-one relation between cluster amplitudes $(t_\mu)_{\mu\in\mathcal{J}}$ and the therewith defined cluster operators $T =\sum_{\mu\in\mathcal{J}}t_{\mu}X_{\mu}$ \cite{rohwedder2013continuous}.
Therefore, we shall denote cluster amplitudes with small letters and the corresponding cluster operators with the respective capital letter. 
Let $\mathcal{V}_K^{(\mathrm{CC})}=\{ t\in \mathbb{R}^{|\mathcal{J}|}~:~\Vert t \Vert_{\mathcal{V}_{K}^{(\mathrm{CC})}}<+\infty \}$ be the (Hilbert) space of cluster amplitudes, where 
\begin{equation*}
\Vert \cdot \Vert_{\mathcal{V}_K^{(\mathrm{CC})}}:\mathbb{R}^{|\mathcal{J}|}\to[0,+\infty];~ t\mapsto\Vert t \Vert_{\mathcal{V}_K^{(\mathrm{CC})}}=\sqrt{\sum_{\mu\in\mathcal{J}}\varepsilon_{\mu}|t_{\mu}|^2}~.
\end{equation*} 
We see that $\Vert \cdot \Vert_{\mathcal{V}_K^{(\mathrm{CC})}}$ is a norm if $\varepsilon_{\mu}> 0$ for all $\mu\in \mathcal J$.
We then refer to $\Vert \cdot \Vert_{\mathcal{V}_K^{(\mathrm{CC})}}$ as the cluster amplitude norm. 
This is guaranteed by assuming a HOMO-LUMO gap, i.e., $\varepsilon_0=\lambda_{N+1}-\lambda_{N}>0$.

Although a HOMO-LUMO gap is very common in electronic structure analysis, it limits the results to a subset of systems. 
For statically correlated systems the Fock operator usually has a degenerate or almost degenerate spectrum, i.e., there exists no HOMO-LUMO gap or it is negligibly small. 
In either case, this yields divergence of the used quasi-Newton method since the HOMO-LUMO gap enters inversely in the approximate Jacobian. 

Formally, the linked CC equations can be defined using the CC function
\begin{equation*}
f_{\mathrm{CC}}:\mathcal{V}_K^{(\mathrm{CC})} \to \big(\mathcal{V}_K^{(\mathrm{CC})}\big)' ;~t \mapsto (\langle  \phi_{\mu}, e^{-T}He^{T}\phi_0  \rangle)_{\mu\in\mathcal{J}}
\end{equation*}
with the energy functional $\mathcal{E}_{\rm CC}:\mathcal{V}_K^{(\mathrm{CC})} \to \mathbb{R};~ ~t \mapsto  \langle\phi_0,e^{-T}He^{T}\phi_0\rangle$.
Consequently, we can write Eqs.~\eqref{eq:LinkedCC} as
\begin{equation*}
\left\lbrace
\begin{aligned}
E^{(\mathrm{CC})}_0&=\mathcal{E}_{\rm CC}(t)~,\\
0&=\langle v,f_{\mathrm{CC}}(t) \rangle~, \quad \text{for all}~v\in \mathcal{V}_K~.
\end{aligned}
\right.
\end{equation*}
This shows that the projected CC method is a nonlinear Galerkin scheme.
A corresponding analysis can be found in \cite{schneider2009analysis}.
\subsection{The Tailored Coupled-Cluster Method}
A major drawback of the projected CC theory is the intractability of statically correlated systems.  
Many attempts have been taken to remedy this impediment but so far no panacea has been found~\cite{RevModPhys.79.291}. 
The TCC method, as an externally corrected CC method, is not based on the Jeziorski--Monkhorst ansatz \cite{RevModPhys.79.291,lyakh2012multireference,kohn2013state}, however, it is still able to compute statically correlated systems with comparable accuracy \cite{veis2016coupled,doi:10.1021/acs.jpclett.6b02912,veis2018full,faulstich2019numerical,antalik2019towards}.
Using a basis splitting approach \cite{piecuch1993state,piecuch1994state,adamowicz1998state,piecuch2010active} it is possible to combine the single-reference CC method with CAS computations \cite{kinoshita2005coupled}.
To that end, the wavefunction is split into two parts: a fixed part imported from a prior CAS calculation and an external part, which is adjusted in the presence of that fixed CAS part.
We use the following basis splitting. 
\begin{definition}
Let $\{\chi_1,...,\chi_K\}\subseteq H^1$ be a set of $L^2$-orthonormal spin-orbitals with $K>N$ and $\phi_0$ the considered reference Slater determinant. 
We define
\begin{align*}
\mathscr{B}_{\mathrm{CAS}} = \lbrace \underbrace{\chi_1,..., \chi_N}_{\text{occupied}}, \underbrace{\chi_{N+1},..., \chi_k}_{\text{unoccupied}}\rbrace~,\quad
\mathscr{B}_{\mathrm{ext}} = \lbrace  \underbrace{\chi_{k+1},...,\chi_K }_{\text{external}}\rbrace
\end{align*}
and furthermore $\mathcal B_\mathrm{CAS} = \{\phi[\mu_1,...,\mu_N] :  \mu_i\in \{1,..., k \}, ~\mu_1<...<\mu_N\} $. 
The corresponding FCI space $\mathcal{H}_{\mathrm{CAS}}$ is then defined as the span of $\mathcal B_\mathrm{CAS}$.
We define $\mathcal{H}_{\mathrm{ext}}$ to be the $L^2$-orthogonal space of $\mathcal{H}_{\mathrm{CAS}}$, i.e., $ \mathcal{H}_K = \mathcal{H}_{\mathrm{CAS}}\oplus \mathcal{H}_{\mathrm{ext}}$.
Analogously, we split the set of excitation-indices $\mathcal{J}$ describing the set of possible excitations, i.e., $\mathcal{J}_{\mathrm{CAS}}=\left\lbrace \mu\in\mathcal{J}:X_{\mu}\phi_0\in \mathcal{H}_{\mathrm{CAS}} \right\rbrace$ and $\mathcal{J}_{\mathrm{ext}}=\left\lbrace \mu\in\mathcal{J}:X_{\mu}\phi_0\notin \mathcal{H}_{\mathrm{CAS}} \right\rbrace$.
\end{definition}
\begin{rmk}
We note that $\mathcal{J}_{\mathrm{ext}}$ does not only contain excitations into states purely excited in $\mathscr{B}_{\mathrm{ext}}$ but also into mixed states, i.e., for $\mu=\binom{A_1,...,A_n}{I_1,...,I_n}$ there exists at least one $l\in\{1,...,n\}$ such that $A_l\in\{k+1,...,K\}$.
\end{rmk}

We highlight that the basis splitting in practice cannot be arbitrary. 
For the correctness of the TCC method it is of utmost importance that $\mathscr{B}_{\mathrm{CAS}}$ covers all statically correlated spin-orbitals. Moreover, $\mathscr{B}_{\mathrm{ext}}$ should only consist of spin-orbitals with dynamic electron correlation. 
A well-chosen basis splitting can be obtained using concepts of quantum information theory as has been introduced in~\cite{szalay2017correlation}.
This caveat will be further discussed in Section \ref{sec:DMRGinQC}.
We also refer to \cite{faulstich2019numerical} for a case study on the N$_2$ molecule illustrating the TCC method's sensitivity to the CAS choice.

Given an intermediately normalized approximate CAS-solution $\phi_{\mathrm{CAS}}$, we can write $\phi_{\mathrm{CAS}}=e^{T^{\mathrm{CAS}}}\phi_0\approx \psi_{\rm CAS}^{(\mathrm{FCI})}$.
The TCC solution is then given by $\psi_*^{(\mathrm{TCC})}=e^{T^{\mathrm{ext}}}e^{T^{\mathrm{CAS}}}\phi_0$, where $T^{\rm ext}$ is obtained by solving the linked TCC equations:
\begin{equation}
\left\lbrace
\begin{aligned}
E_0^{(\mathrm{TCC})}
&=
\langle
\phi_0,e^{-T^{\mathrm{CAS}}}e^{-T^{\mathrm{ext}}}He^{T^{\mathrm{CAS}}}e^{T^{\mathrm{ext}}}\phi_0
\rangle~,\\
0
&=
\langle
\phi_{\mu},e^{-T^{\mathrm{CAS}}}e^{-T^{\mathrm{ext}}}He^{T^{\mathrm{CAS}}}e^{T^{\mathrm{ext}}}\phi_0
\rangle~,\quad \mu\notin\mathcal{J}_{\mathrm{CAS}}~.\\
\end{aligned}
\right. 
\label{eq:Linked-TCC-eq}
\end{equation}
We emphazise that for the TCC method, the CAS-solution $\phi_{\mathrm{CAS}}$ and therewith $T^{\rm CAS}$ is fixed. 
Similar to the analysis in \cite{schneider2009analysis}, a useful measure for the dynamical correction is a weighted $\mathit{l}^2$-norm of the external cluster amplitudes. 
Let 
\begin{equation*}
\mathcal{V}_{\mathrm{ext}}=\{ t\in\mathbb{R}^{|\mathcal{J}_{\mathrm{ext}}|}~:~\Vert t \Vert_{\mathcal{V}_{\mathrm{ext}}}<+\infty\}
\end{equation*} 
be the space of external cluster amplitudes, where 
\begin{equation*}
\Vert \cdot \Vert_{\mathcal{V}_{\mathrm{ext}}}:\mathbb{R}^{|\mathcal{J}_{\mathrm{ext}}|}\to [0,+\infty];\quad t\mapsto\Vert t \Vert_{\mathcal{V}_{\mathrm{ext}}}=\sqrt{\sum_{\mu\in\mathcal{J}_{\mathrm{ext}}}\varepsilon_{\mu}|t_{\mu}|^2}~.
\end{equation*}
The map $\Vert \cdot \Vert_{\mathcal{V}_{\mathrm{ext}}}$ is a norm if $\varepsilon_{\mu}>0$ for all $\mu\in\mathcal{J}_{\mathrm{ext}}$.
Assumptions on the considered systems to ensure such structure will be elaborated in Section~\ref{subsec:LocalUniqueness}.
Using this framework we can define the $N$-electron TCC function as follows.

\begin{definition}
\label{def:Tcc_function}
Let $K,\,N\in\mathbb{N}$ with $K>N$ be fixed, $\mathscr{B}=\{\chi_1,...,\chi_K\}\subseteq H^1$ a set of $L^2$-orthonormal spin-orbitals and $\phi_0\in\mathcal{H}_K$ the considered reference state.
Further, assume the splitting $\mathscr{B} = \mathscr{B}_{\mathrm{CAS}}\dot{\cup}\mathscr{B}_{\mathrm{ext}}$ of $\mathscr{B}$ and the CAS-solution $\phi_{\mathrm{CAS}}=e^{T^{\mathrm{CAS}}}\phi_0$, with corresponding amplitudes $t^{\rm CAS}=(t_{\mu}^{\rm CAS})_{\mathcal{J}_{\rm CAS}}$.
We define the TCC function
\begin{equation*}
f(\,\cdot\,;t^{\mathrm{CAS}}):\mathcal V_\mathrm{ext}\rightarrow\big(\mathcal V_\mathrm{\mathrm{ext}}\big)'\,;\quad t \mapsto f(t;t^{\mathrm{CAS}})~,
\end{equation*}
where $(f(t;t^{\mathrm{CAS}}))_{\mu}=\langle  \phi_{\mu}, e^{-T^{\mathrm{CAS}}}e^{-T}He^{T}e^{T^{\mathrm{CAS}}}\phi_0  \rangle$ for $\mu\in\mathcal{J}_{\mathrm{ext}}$.
In addition, let the TCC-energy functional be given by
\begin{equation*}
\mathcal E (t;t^{\mathrm{CAS}}) = \langle  \phi_{0}, e^{-T^{\mathrm{CAS}}}e^{-T}He^{T}e^{T^{\mathrm{CAS}}}\phi_0  \rangle~.    
\end{equation*}
\end{definition}

Using the TCC function, the linked TCC equations \eqref{eq:Linked-TCC-eq} become
\begin{equation*}
\left\lbrace
\begin{aligned}
 E_0^{(\mathrm{TCC})} &= \mathcal E(t;t^{\mathrm{CAS}})~,\\ 
0&=\langle v,f(t;t^{\mathrm{CAS}}) \rangle~, \quad \forall ~v\in \mathcal{V}_{\mathrm{ext}}~.\\
\end{aligned}
\right.
\end{equation*}
This formulation resembles the single-reference CC method. 
Indeed, $f(t;t^{\mathrm{CAS}}) = P_{\mathcal{V}_{\rm ext}}f_\mathrm{CC}(t \oplus t^{\mathrm{CAS}})$ with the orthogonal projection  $P_{\mathcal{V}_{\rm ext}}$ onto $ \mathcal{V}_{\rm ext}$, relates the TCC function to the classical CC function in Eq.~\eqref{eq:LinkedCC}.
Note that the CAS-part of the cluster amplitudes is still fixed.
Despite this close connection to the CC method, we shall see that the TCC scheme differs heavily from the single-reference CC method in its computational performance and analysis.
\subsection{Entropy based CAS choice}
\label{sec:DMRGinQC}
We start this section by noting that any Slater determinant can be uniquely described by an occupation tensor $\mathbf{e}^{\mathsf{m}_1} \otimes ... \otimes \mathbf{e}^{\mathsf{m}_K}$, where $ \mathbf{e}^0 = (1,~0)^T, \mathbf{e}^1 = (0,~1)^T \in \mathbb{R}^2 $.
This identification is part of the second quantization~\cite{helgaker2014molecular} and is in fact an isometric isomorphism (see the Jordan--Wigner transformation \cite{nielsen2005fermionic}).
Consequently, we can interpret any real wavefunction as an element in the $2^K$ dimensional linear space $\mathcal{W}_K =\bigotimes_{i=1}^K \mathbb{R}^2$ with given bais $\{\phi_{\mathsf{m}} = \mathbf{e}^{\mathsf{m}_1} \otimes ... \otimes \mathbf{e}^{\mathsf{m}_K} ~:~ \mathsf{m}_i\in\{0,1\}\}$.
Given a low-rank DMRG solution $\psi_{\rm DMRG}$ on $\mathcal{W}_K$, i.e., $\psi_{\rm DMRG}=\sum_{\mathsf{m}=1}^K c_{\mathsf{m} } \phi_{\mathsf{m} }$, we introduce the quantum information theory concepts used to chose a CAS.
We start by considering the $i$-mode matricization $\textbf{U}{[i]}\in\mathbb{R}^{ 2^{K-1}\times 2}$ of the solution tensor $\psi_{\rm DMRG}$, i.e., the matrix obtained form $\psi_{\rm DMRG}$ by transforming the basis elements $\phi_{\mathsf{m}}$ by taking $\mathsf{m}_i$ as row index and all remaining indices as compound column index.
We introduce the elementwise notation $U[i]_{(\mathsf{m}_1,...,\not\mathsf{m}_i,...\mathsf{m}_K),(\mathsf{m}_i)}$, where $\cancel{\mathsf{m}_i}$ means that $\mathsf{m}_i$ is removed from the binary string $\mathsf{m}$ and all remaining indeces are combined to one compound index. 
We then compute the single-orbital entropy for the $i$-mode matricization denoted $s(i)$, i.e., $ s(i) = -\mathrm{Tr}(\textbf{D}[i] \mathrm{ln}\textbf{D}[i])\in [0,\ln(2)]$, where $\textbf{D}[i]=\textbf{U}[i]^T \textbf{U}[i]\in\mathbb{R}^{ 2\times 2}$ is the single-orbital density matrix.
Based on Szalay {\it et al.} \cite{szalay2017correlation}, the single-orbital entropy can be used to describe the degree of electron correlation, i.e., a large value of $s(i)$ indicates static correlations. 
However, since the electron correlation is a two particle effect, we need to measure the information flow for all possible electron pairs. 
This is done via the mutual information:
We start by computing the two-orbital entropy $s(i,j)$.
Similarly to the single-orbital entropy $s(i)$, the two-orbital entropy $s(i,j) = -\mathrm{Tr}(\textbf{D}[i,j] \mathrm{ln}\textbf{D}[i,j])\in [0,\ln(4)]$ where $\textbf{D}[i,j]\in\mathbb{R}^{ 4\times 4}$ is the two-orbital density matrix obtained from  $U[i,j]_{(\mathsf{m}_i,\mathsf{m}_j)(\mathsf{m}_1,...,\not{\mathsf{m}_i},...,\not{\mathsf{m}_j}...\mathsf{m}_K)}$.
Given the single- and two-orbital entropies, we can compute the mutual information, $I(i,j) = s(i)+s(j)-s(i,j)$ for $i,j\in {1,....K}$.
This quantifies the electron correlations between orbital $i$ and $j$ as they are embedded in the whole system \cite{Rissler2006}.
The large values of $I(i,j)$ describe static correlations while the small matrix elements stand for the dynamic correlation. 
In certain cases, the decreasingly ordered values of $I(i,j)$ show a jump, which clearly distinguishes a set of statically correlated orbitals, and suggests a basis splitting at this jump. 
However, general mutual information profiles do not need to show such behavior. 
Then the {\it a priori} thresholds $\underline{s}$ and $\underline{n}$ are introduced to identify orbitals with $s(i)>\underline{s}$ and $I(i,j)>\underline{n}$.
It is these orbitals that are then used to define $\mathscr{ B}_{\mathrm{CAS}}$ and therewith the basis splitting.
In practice, $\underline{s}$ and $\underline{n}$ are systematically lowered until convergence of the DMRG-TCC method is reached.
This approach is heuristic but provides an efficient tool for obtaining well-chosen $\mathscr{B}_{\mathrm{CAS}}$ and $\mathscr{B}_{\mathrm{ext}}$, which is essential for the TCC method's success.
We highlight that the above procedure is feasible for larger systems since the used quantities are qualitatively very robust with respect to the bond-dimension, i.e., a CAS choice can be obtained from a low rank calculation on $\mathcal{H}_K$~\cite{faulstich2019numerical}.
For more details and numerical investigations on the CAS choice we refer the reader to \cite{faulstich2019numerical}.

\section{Analysis of the TCC Method}
\label{Sec:Analysis}
We focus here on the mathematical analysis of the TCC method for a finite spin-orbital set, i.e., $K<\infty$.
Several caveats of the limit process $K\to\infty$ are subsequently addressed, but a full investigation is relegated to future work.

First, we show the consistency of the TCC method, in the sense that exact solutions of the Schr\"odinger equation are reproduced.
We denote $\psi_* = e^{T_*^{\mathrm{FCI}}} \phi_0$ the exact solution on $\mathcal H_K$. 
We split the amplitudes such that $T_*^{\mathrm{FCI}} = T_*^{\mathrm{CAS}} + T_*^{\mathrm{ext}}$ with $t_*^{\mathrm{CAS}}\in\mathcal{V}_{\mathrm{CAS}}$ and $t_*^{\mathrm{ext}}\in\mathcal{V}_{\mathrm{ext}}$.
\begin{theorem}
Let $E$ be any eigenvalue of $H$ and assume $\psi_*$ satisfies the Schr\"odinger equation. 
Then $f(t_*^{\mathrm{ext}};t_*^{\mathrm{CAS}})=0$ and $E = \mathcal E(t_*^{\mathrm{ext}};t_*^{\mathrm{CAS}})$.
\label{Thm:Implication}
\end{theorem}
\begin{proof}
Let $\mu\in \mathcal J_\mathrm{ext}$ and choose $\psi'= e^{-(T_*^{\mathrm{ext}})^\dagger} e^{-(T_*^{\mathrm{CAS}\dagger})} \phi_\mu \in\mathcal{H}_K$. 
By assumption
\begin{align*}
0 = \langle \psi',(H-E) \psi_*\rangle  =\langle \phi_\mu,  e^{-T_*^\mathrm{CAS}} e^{-T^{\mathrm{ext}}_*} (H-E) e^{T^{\mathrm{ext}}_*} e^{T_*^\mathrm{CAS}} \phi_0 \rangle = (f(t^{\mathrm{ext}}_*;t_*^{\mathrm{CAS}}))_\mu~.  
\end{align*}
Inserting instead $\psi' = e^{-(T_*^{\mathrm{ext}})^\dagger} e^{-(T_*^{\mathrm{CAS}\dagger})} \phi_0 \in  \mathcal{H}_K$ gives $E= \mathcal E(t^{\mathrm{ext}}_*;t_*^{\mathrm{CAS}})$.
\end{proof}
\begin{rmk}

An important observation is that tailoring the CC method with a FCI solution on the CAS, i.e., a solution that corresponds to $t^{\rm CAS}_{\rm FCI}$, does not necessarily reproduce the FCI solution on $\mathcal{H}_K$.
More precisely, let $f(t^{\rm ext};t^{\rm CAS}_{\rm FCI})= 0$ then $\psi_*^{\rm (TCC)} = e^{T^{\rm ext}}e^{T^{\rm CAS}_{\rm FCI}}\phi_0$ is not necessarily a minimizer of Eq.~\eqref{eq:Rayleigh-Ritz variational principle} and does therewith in general not fulfill the Schr\"odinger equation.
However, Theorem~\ref{Thm:Implication} shows that $f(t^{\mathrm{ext}};t^{\mathrm{CAS}})=0$ is a necessary condition for $\psi = e^{T^{\rm ext}} e^{T^{\mathrm{CAS}}} \phi_0$ to solve the Schr\"odinger equation on $\mathcal H_K$. 
In the continuous formulation of the traditional CC theory, equivalence has been proven in \cite{rohwedder2013continuous} see Theorem 5.3.
Equivalence for the projected CC method has been shown in \cite{laestadius2019coupled} see Section~2.2., using~\cite{monkhorst1977calculation}.
\end{rmk}

We emphasize that the CAS part $T_*^{\mathrm{CAS}}$ of the exact cluster operator $T_*^{\mathrm{FCI}}$ is not equal to the cluster operator that corresponds to the FCI solution on $\mathcal{H}_{\mathrm{CAS}}$. 
The CAS amplitudes $((t_*^{\mathrm{CAS}})_\mu)_{\mu\in\mathcal{J}_{\mathrm{CAS}}}$ on $\mathcal{H}_{K}$ are solutions of equations that depend on the external amplitudes. 
The FCI solution $\psi_{\mathrm{CAS}}^{(\mathrm{FCI})}=e^{T_{\mathrm{FCI}}^{\mathrm{CAS}}}\phi_0$ on $\mathcal{H}_{\mathrm{CAS}}$, however, depends on the Hamilton operator projected onto the CAS space. 
Hence, in general $T_{\mathrm{FCI}}^{\mathrm{CAS}}\neq T_*^\mathrm{CAS}$.
\begin{rmk}
    Theorem~\ref{Thm:Implication} does not imply \emph{local uniqueness} of $t_*\in\mathcal{V}_\text{ext}$, even if $t_*^\text{FCI}$ is locally unique.
\end{rmk}

Throughout Subsection \ref{subsec:LocalUniqueness} we consider a fixed and sufficiently good CAS solution, i.e., $\phi_{\rm CAS}\approx \psi_{\rm CAS}^{\rm (FCI)}\approx P_{\mathcal{V}_{\rm CAS}}\psi_*$.
As a consequence we will simplify the notation by neglecting the parametric dependency of $f$ and $\mathcal{E}$ on $t^{\mathrm{CAS}}$.
We also highlight that the following analysis holds for any TCC scheme, but in particular for TNS-TCC schemes like the DMRG-TCCSD method.

\subsection{Local Uniqueness and Residual Bounds}
\label{subsec:LocalUniqueness}
The single-reference CC method, as well as the considered TCC method, are formulated as nonlinear Galerkin schemes.
This suggests the use of Zarantonello's lemma \cite{zaidler1990nonlinear} to characterize local uniqueness and residual bounds.
This is in line with previous studies on single-reference CC methods \cite{schneider2009analysis,rohwedder2013error,laestadius2018analysis}.
We state without proof:
\begin{lemma}[Local Version of Zarantonello's lemma \cite{zaidler1990nonlinear}]
\label{Th:Zarantonello}
Let $g:X \to X'$ be a map between a Hilbert space $(X,\langle\cdot,\cdot\rangle,\Vert \cdot \Vert)$ and its dual $X'$, and let $x_*\in B_{\delta}$ be a root, $g(x_*)=0$, where $B_{\delta}$ is an open ball of radius $\delta$ around $x_*$.
Assume that $g$ is Lipschitz continuous and locally strongly monotone in $B_{\delta}$ with constants $L > 0$ and $\gamma>0$, respectively. 

Then  
the root $x_*$ is unique in $B_{\delta}$. 
Indeed, there is a ball $C_{\varepsilon}\subset X'$ with $0\in C_{\varepsilon}$ such that the solution map $g^{-1}:C_{\varepsilon}\to X$ exists and is Lipschitz continuous, implying that the equation $g(x_*+x)=y$ has a unique solution $x=g^{-1}(y)-x_*$, depending continuously on $y$, with norm $\Vert x \Vert\leq \delta$.
Moreover, let $X_d\subset X$ be a closed subspace such that $x_*$ can be approximated sufficiently well, i.e., the distance $d(x_*,X_d)$ is sufficiently small. 
Then, the projected problem $g_d(x_d)=0$ has a unique solution $x_d\in X_d\cap B_{\delta}$ and 
\begin{equation*}
\Vert x_*-x_d \Vert\leq \frac{L}{\gamma} d(x_*,X_d)~.
\end{equation*}
\end{lemma}

We emphasize that the above theorem depends strongly on the topology of the considered Hilbert space.
We already made the particular choice of $\Vert\cdot\Vert_{\mathcal{V}_{\mathrm{ext}}}$ to measure the dynamical correction. 
This is motivated by the fact that $(\varepsilon_\mu)_{\mu\in \mathcal{J}_{\mathrm{ext}}}$ is computationally accessible.
A major difference between the presented analysis and the single-reference CC case \cite{schneider2009analysis,rohwedder2013continuous,rohwedder2013error} is that the assumption of a HOMO-LUMO gap is no longer reasonable.
In the context of the TCC method it is assumed that $\mathscr B_\mathrm{CAS}$ and $\mathscr B_\mathrm{ext}$ are chosen such that 
$\lambda_{k+1}- \lambda_k>0$. 
We therefore introduce the CAS-ext gap between $\lambda_k$ and $\lambda_{k+1}$.
In analogy to previous literature on analysis of the CC theory, we denote the CAS-ext gap by $\varepsilon_0=\lambda_{k+1}-\lambda_{k}$. 
The assumption of a CAS-ext gap is reasonable under the assumption that $\mathcal{H}_\mathrm{CAS}$ captures all strong correlation such that the (one-particle) Fock operator's degenerate eigenstates are in the CAS.

Besides the single-particle spectral gap condition, we note that the Fock operator $F$ corresponds to a Hamilton operator with a particular potential $V_F$ in Eq.~\eqref{Eq:Hgen}. 
Consequently, with $V=V_F$ in Eq.~\eqref{eq:Gard} we assume 
\begin{equation}
\langle \psi,(F+e)\psi\rangle \geq c \Vert \psi\Vert_{H^1}^2~,\quad \forall \psi\in H^1~.
\label{eq:Fgard}
\end{equation}
For a further discussion on spectral gap and Gårding inequalities in CC theories we refer to~\cite{laestadius2019coupled}. Moreover, in agreement with Section \ref{sec:SE}, we suppose
\begin{equation}
|\langle \tilde\psi,F\psi\rangle |\leq C \Vert \tilde \psi\Vert_{H^1} \Vert \psi\Vert_{H^1}~,\quad\forall \psi,\tilde{\psi}\in H^1~.
\label{eq:Fbdd}
\end{equation}
One of the main assumption of this article can then be summarized:

\textbf{Assumption} (A). \textit{For the Fock operator $F$, Eqs.~\eqref{eq:Fgard} and \eqref{eq:Fbdd} hold and there exists a CAS-ext gap $\varepsilon_0= \lambda_(k+1) - \lambda_k>0$.}
\begin{rmk}
\label{rmk:extended_cas_ext_gap}
Note that a gap assumption between $\lambda_N$ and $\lambda_{k+1}$ is also possible, i.e., $\tilde\varepsilon_0=\lambda_{k+1}-\lambda_{N}$.
We shall refer to this as the extended CAS-ext gap.
The difference to $\varepsilon_0$ is that $\tilde\varepsilon_0$ is directly proportional to the size of the CAS, i.e., choosing a large CAS yields a large $\lambda_{k+1}$ and therewith a large value of $\tilde\varepsilon_0$. 
Consequently, this connects the following norm estimates with the CAS.
We point out that every following statement holds true for either gap condition, however, the constants involved may differ. 
\end{rmk}

The main argument for considering $\varepsilon_0$ (or $\tilde \varepsilon_0$) is that the following analysis holds not only for ground-state approximation schemes but also for excited state approximations, which is a major difference to the previous analyses of single-reference CC methods \cite{schneider2009analysis,rohwedder2013continuous,rohwedder2013error,laestadius2018analysis}.
In the TCC scheme, the single-reference CC method is used to add a dynamical correction to $\phi_{\mathrm{CAS}}\in\mathcal{H}_{\mathrm{CAS}}$ on the external space $\mathcal{H}_{\mathrm{ext}}$, i.e., it captures dynamical correlations between orbitals in $\mathcal{H}_{\mathrm{ext}}$ as well as dynamical correlations between orbitals in $\mathcal{H}_{\mathrm{CAS}}$ and $\mathcal{H}_{\mathrm{ext}}$.
This correction can be done for any wavefunction $\phi_{\mathrm{CAS}}\in\mathcal{H}_{\mathrm{CAS}}$, in particular also for approximations of excited states in $\mathcal{H}_{\mathrm{CAS}}$.
We emphasize that correlations between orbitals in $\mathscr{B}_{\mathrm{ext}}$ and $\mathscr{B}_{\mathrm{CAS}}$ are not considered when computing $\phi_{\mathrm{CAS}}$, which introduces a methodological error to the method \cite{faulstich2019numerical}.

Note that Assumption (A) is an assumption on the single-particle spectrum. 
This allows us to establish $\varepsilon_{\mu} > \varepsilon_0$ for all $\mu\in\mathcal{J}_{\rm ext}$, however, it does not necessarily imply $\varepsilon_{\sigma} \leq \varepsilon_{\mu}$ for $\sigma \in \mathcal{J}_{\rm CAS}$ and $\mu \in \mathcal{J}_{\rm ext}$. Thus, under Assumption~(A) we might not have a spectral gap in the $N$-particle space.

Next, we introduce the Fock norm on $\mathcal{H}_{\rm ext}$.

\begin{definition}
The map $
\Vert \cdot \Vert_F : \mathcal{H}_{\rm ext} \to \mathbb{R}_+$ is given by 
$\phi \mapsto \sqrt{\langle \phi , (F-\Lambda_0) \phi  \rangle} $.
\end{definition}

\begin{lemma}
\label{Lemma:H1}
Suppose Assumption (A), then $\Vert \phi \Vert_F=\sqrt{\langle \phi,(F-\Lambda_0)\phi\rangle}$ is a norm on $\mathcal H_{\mathrm{ext}}$ and 
\begin{equation}
\langle T\phi_0, ( F - \Lambda_0) T\phi_0 \rangle 
\geq 
\eta \Vert T\phi_0\Vert_{H^1}^2~,\quad \forall t\in\mathcal V_{\mathrm{ext}}~,
\label{eq:GapGaarding}
\end{equation}
where $\eta>0$ is defined in the proof. 
Moreover, $\Vert \cdot \Vert_F$ is equivalent to $\Vert \cdot\Vert_{H^1}$ on $\mathcal H_{\mathrm{ext}}$.
\end{lemma}
\begin{proof}
The assumption of a Gårding inequality of the Fock operator and a spectral gap (Eq.~\eqref{eq:Fbdd}) imply \eqref{eq:GapGaarding}. 
The derivation is given by Lemma 11 in \cite{laestadius2018analysis} and is here included to highlight the importance of a CAS-ext gap.
Before starting the proof, we note that Eq.~\ref{eq:Fgard} implies $e\geq \Lambda_0$, since $\Lambda_0$ is the smallest eigenvalue of $F$ in $\mathcal{H}_K$.
Then we set $q=\varepsilon_0/(\varepsilon_0+\Lambda_0+e)>0$ and $\eta = qc$, where $e,c$ are the constants from the Gårding inequality~\eqref{eq:Fgard}.
Assumption (A) yields $\langle T\phi_0, ( F - \Lambda_0) T\phi_0 \rangle \geq \varepsilon_0 \Vert T\phi_0 \Vert_{L^2}^2 $, for $t\in\mathcal{V}_{\mathrm{ext}}$.
The Gårding inequality~\eqref{eq:Fgard} implies
\begin{equation*}
\begin{aligned}
&\langle T\phi_0, ( F - \Lambda_0) T\phi_0 \rangle\\
&=
q \langle T\phi_0, ( F +e) T\phi_0 \rangle
-q \langle T\phi_0, (\Lambda_0+e) T\phi_0 \rangle
+(1-q) \langle T\phi_0, ( F - \Lambda_0) T\phi_0 \rangle\\
&\geq
qc \Vert T\phi_0\Vert^2_{H^1}+((1-q)\varepsilon_0-q( \Lambda_0+e)) \Vert T\phi_0\Vert^2_{L^2}
=
\eta \Vert T\phi_0\Vert^2_{H^1}
~.
\end{aligned}
\end{equation*}
Therefore $\Vert \phi \Vert_F=0$ if and only if $\phi=0$. 
The self-adjointness of $F$ gives the triangle inequality and the homogeneity follows immediately. 
Hence, $\Vert \phi \Vert_F$ is a norm.
The proof is completed by noting that 
Eq.~\eqref{eq:GapGaarding} and the boundedness of $F$ (Eq.~\eqref{eq:Fbdd}) yield the equivalence of $\Vert \cdot \Vert_F$ and $\Vert \cdot\Vert_{H^1}$ on $\mathcal H_{\mathrm{ext}}$.
\end{proof}

\begin{proposition}
\label{prop:Focknorm_ExtNorm}
For $t\in\mathcal{V}_{\rm ext}$ $
\Vert t \Vert_{\mathcal{V}_{\rm ext}}=
\Vert T\phi_0 \Vert_F$, and in particular $\Vert t \Vert_{\mathcal{V}_{\rm ext}} \sim \Vert T\phi_0 \Vert_{H^1}$.
\end{proposition}

\begin{rmk}
Note that the spectral (CAS-ext) gap assumption of $F$ gives 
\begin{equation*}
\begin{aligned}
\Vert T \phi_0 \Vert_F^2 &=\langle T\phi_0 , (F-\Lambda_0) T\phi_0 \rangle \geq \varepsilon_0 \Vert T\phi_0 \Vert_{L^2}^2~,
\end{aligned}
\end{equation*}
which is the same as the direct estimate
$\Vert t \Vert_{\mathcal{V}_{\rm ext}}^2  = \sum_{\mu \in \mathcal I_\mathrm{ext}} \varepsilon_\mu t_\mu^2  \geq \varepsilon_0 \Vert t\Vert_2^2$. 
This makes the Fock norm natural in the following analysis.
\end{rmk}

Two useful facts regarding the Fock operator and excitation operators are stated in the following lemma (for a proof see~\cite{helgaker2014molecular}).
\begin{lemma}
Let $F$ be the Fock operator, $\mu=\binom{A_1,...,A_{\vert \mu\vert}}{I_1,...,I_{\vert \mu\vert}}$ and $T=\sum_{\mu\in\mathcal{J}}t_\mu X_\mu$. 
Then 
\begin{equation*}
[F,X_{\mu}]=\sum_{j=1}^{{\vert \mu\vert}}(\lambda_{A_j}-\lambda_{I_j})X_{\mu}=\varepsilon_{\mu}X_{\mu}
\quad {\rm and} \quad
e^{-T}Fe^{T}=F+[F,T]~.
\end{equation*}
\label{Lemma:FockOperator}
\end{lemma}

\begin{proof}[Proof of Proposition~\ref{prop:Focknorm_ExtNorm}]
Let $t\in \mathcal{V}_{\rm ext}$, we find by means of Lemma~\ref{Lemma:FockOperator}
\begin{equation*}
\begin{aligned}
\Vert T\phi_0 \Vert_F^2 
&= \langle T \phi_0 , (F-\Lambda_0) T\phi_0 \rangle
= \sum_{\mu,\nu\in\mathcal{J}_{\rm ext}} t_{\mu} t_{\nu} \langle \phi_{\mu} , (F-\Lambda_0) \phi_{\nu}  \rangle\\
&= \sum_{\mu,\nu\in\mathcal{J}_{\rm ext}} t_{\mu}  t_{\nu} \langle \phi_{\mu} , [F,X_{\nu}] \phi_0  \rangle
= \sum_{\mu\in\mathcal{J}_{\rm ext}} t_{\mu}^2 \varepsilon_{\mu} 
= \Vert t \Vert_{\mathcal{V}_{\rm ext}}^2~.
\end{aligned}
\end{equation*}
$~$
\end{proof}

\begin{rmk}
\label{rmk:OnTheStructureOfFockOperator}
The first formula of Lemma~\ref{Lemma:FockOperator} uses the fact that $F$ is diagonal, i.e., a finite $K$.  However, the fact that $[F,X_\mu]$ is a cluster operator can be proven using only the $F$-orthogonality of an occupied $\chi_I$ and an unoccupied $\chi_A$.  Thus, while in the infinite-dimensional case the first statement certainly fails due to the continuous spectrum, it is reasonable to expect that the second statement still stands.
\end{rmk}

\begin{theorem}
Under Assumption (A) the norm equivalence $\Vert T \Vert_{\mathcal{B}\left(H^1\right)}\sim \Vert  t \Vert_{\mathcal{V}_{\mathrm{ext}}}$ holds for $t\in \mathcal{V}_{\mathrm{ext}}$ .
\label{th:OperatorNormEquiv}
\end{theorem}

\noindent
To show this we first prove the following lemma.
\begin{lemma}
\label{lemma:FockEstimate66}
Let $\nu\in\mathcal{J}_{\mathrm{ext}}$ and $\alpha,\mu\in\mathcal{J}$ with $|\alpha|,~|\mu|\leq |\nu|$ and $\langle \phi_{\nu}, X_{\alpha}\phi_{\mu}  \rangle\neq 0$. Then there exists a constant $C\geq 0$ such that
\begin{align*}
i)\quad\frac{\varepsilon_{\nu}}{\varepsilon_{\mu}}\leq C\varepsilon_{\alpha}\,,\text{ if }\alpha,~\mu\in\mathcal{J}_{\mathrm{ext}}\qquad\quad
ii)&\quad \varepsilon_{\nu}\leq C\varepsilon_{\alpha}\,,\text{ if }\alpha\in \mathcal{J}_{\mathrm{ext}}\text{ and }\mu\in\mathcal{J}_{\mathrm{CAS}}~.
\end{align*}
\end{lemma}
\begin{proof}
Set $\delta= (\lambda_{k+1}+\lambda_k)/2$ and define $\overline{\lambda}_{\nu}=\max\{\lambda_{A_j}: \, j=1,...,|\nu|\}-\delta$, which is well-defined since $K$ is finite.
We first demonstrate, following Lemma 4.14 in \cite{schneider2009analysis}, for all $\nu\in\mathcal{J}_{\mathrm{ext}}$ there exists a $C>0$ such that
\begin{equation}
C^{-1}\varepsilon_{\nu}\leq \overline{\lambda}_{\nu}\leq \varepsilon_{\nu}~.
\label{eq:lambdaIneq}
\end{equation}
Let $\nu\in\mathcal{J}_{\mathrm{ext}}$. It is immediate that $\varepsilon_0^{-1}\geq \varepsilon_{\nu}^{-1}$. 
From the definition of $\overline{\lambda}_{\nu}$, we conclude 
\begin{equation*}
\varepsilon_{\nu} = \sum_{j=1}^{|\nu|} (\lambda_{A_j} - \lambda_{I_j}) \leq N(\overline{\lambda}_{\nu}-(\lambda_1-\delta))~.
\end{equation*}
Since $\overline{\lambda}_{\nu}\geq \lambda_{k+1}- \delta$ it follows $\overline{\lambda}_{\nu} \geq\varepsilon_0/2$, which is equivalent to $(2\overline{\lambda}_{\nu})^{-1}\leq \varepsilon_0^{-1}$. 
This implies $|\lambda_1-\delta|\leq 2|\lambda_1-\delta|\overline{\lambda}_{\nu}/\varepsilon_0$. 
Thus, 
\begin{equation*}
N^{-1}\varepsilon_{\nu} \leq  \overline{\lambda}_{\nu} + |\lambda_1-\delta|	 \leq (1 +  2|\lambda_1-\delta|/\varepsilon_0 )\overline{\lambda}_{\nu}~,
\end{equation*}
which proves the first inequality of Eq.~\eqref{eq:lambdaIneq}. 
For the second inequality we define $\lambda_{A_{j_*}} =~\max\{\lambda_{A_j}: \, j=1,...,|\nu|\}$ and note that	$\varepsilon_\nu \geq \lambda_{A_{j_*}} - \lambda_{I_{j_*}} \geq \lambda_{A_{j_*}}  -\delta = \overline{\lambda}_\nu$.
We now prove the lemma considering three cases:
\begin{itemize}
\item[i)] Let $\alpha,~\mu\in\mathcal{J}_{\mathrm{ext}}$ and $\overline{\lambda}_{\alpha}\geq \overline{\lambda}_{\mu}$. Then $\overline{\lambda}_{\alpha}= \overline{\lambda}_{\nu}$ and we estimate
\begin{equation*}
\frac{\varepsilon_{\nu}}{\varepsilon_{\mu}}\leq \frac{C\overline{\lambda}_{\nu}}{\varepsilon_{0}}=\frac{C}{\varepsilon_{0}}\overline{\lambda}_{\alpha}\leq \frac{C}{\varepsilon_{0}}\varepsilon_{\alpha}~.
\end{equation*}
		\item[ii)]  Let $\alpha,~\mu\in\mathcal{J}_{\mathrm{ext}}$ and $\overline{\lambda}_{\alpha}\leq \overline{\lambda}_{\mu}$. Then $\overline{\lambda}_{\mu}= \overline{\lambda}_{\nu}$ and using $(2\overline{\lambda}_{\alpha})^{-1}\leq \varepsilon_0^{-1}$ we obtain
		\begin{equation*}
		\frac{\varepsilon_{\nu}}{\varepsilon_{\mu}}\leq \frac{C\overline{\lambda}_{\nu}}{\overline{\lambda}_{\mu}}=\frac{2\overline{\lambda}_{\alpha}}{2\overline{\lambda}_{\alpha}}C\leq \frac{2C}{\varepsilon_0}\varepsilon_{\alpha}~.
		\end{equation*}
		\item[iii)] Let $\alpha\in \mathcal{J}_{\mathrm{ext}}$ and $\mu\in\mathcal{J}_{\mathrm{CAS}}$. Then $\overline{\lambda}_{\alpha}= \overline{\lambda}_{\nu}$ and $\varepsilon_{\nu}\leq C \overline{\lambda}_{\nu} = C\overline{\lambda}_{\alpha}\leq C\varepsilon_{\alpha}$.
	\end{itemize}
\end{proof}
\begin{proof}[Proof of Theorem \ref{th:OperatorNormEquiv}]
Proposition~\ref{prop:Focknorm_ExtNorm} implies the inequality $\Vert  t \Vert_{\mathcal{V}_{\mathrm{ext}}}\lesssim \Vert T\phi_0\Vert_{H^1}\leq \Vert T\Vert_{\mathcal{B}(H^1)} \Vert\phi_0\Vert_{H^1}$.
	Consequently, it remains to show that $\Vert T\psi\Vert_{H^1}\leq C\Vert  t \Vert_{\mathcal{V}_{\mathrm{ext}}} \Vert \psi\Vert_{H^1}$ for $\psi\in \text{span}\{\phi_0\}^{\perp}$ (in the $L^2$-sense).
	Let $\psi=\sum_{\mu\in\mathcal{J}}s_{\mu}\phi_{\mu}=S\phi_0\in\mathcal{H}_K$, $T = \sum_{\alpha\in\mathcal{J}_{\mathrm{ext}}}t_{\alpha}X_{\alpha}$ and $s=(s_{\mu})_{\mu\in\mathcal{J}}$, where we assume without loss of generality that $(s_{\mu})_{\mu\in\mathcal{J}}=((s_{\mu})_{\mu\in\mathcal{J}_{\mathrm{ext}}},(s_{\mu})_{\mu\in\mathcal{J}_{\mathrm{CAS}}})$.
	Note that the product $TS$ is an excitation operator with cluster amplitudes in $\mathcal{V}_{\mathrm{ext}}$.
	Hence, Proposition~\ref{prop:Focknorm_ExtNorm} yields
	\begin{equation}
	\begin{aligned}
	\Vert T\psi\Vert^2_{H^1}
	=
	\Vert TS\phi_0\Vert^2_{H^1}
	&\sim 
	\Vert (\langle \phi_{\nu},TS\phi_0  \rangle)_{\nu\in\mathcal{J}_{\mathrm{ext}}}\Vert^2_{\mathcal{V}_{\mathrm{ext}}}
	=
	\Vert (\langle \phi_{\nu},T\psi  \rangle)_{\nu\in\mathcal{J}_{\mathrm{ext}}}\Vert^2_{\mathcal{V}_{\mathrm{ext}}}\\
	&=
	\sum_{\nu\in\mathcal{J}_{\mathrm{ext}}}\Big(
	\varepsilon_{\nu}^{1/2}|\sum_{\alpha\in\mathcal{J}_{\mathrm{ext}}}\sum_{\mu\in\mathcal{J}}t_{\alpha}s_{\mu}\langle \phi_{\nu}, X_{\alpha} \phi_{\mu}  \rangle|
	\Big)^2~.
	\label{eq:estimationSchneider4.15}
	\end{aligned}
	\end{equation}

	We now define $A=\left(
	\langle \phi_{\nu} , T \phi_{\mu}  \rangle
	\right)_{\nu\in\mathcal{J}_{\mathrm{ext}},\mu\in\mathcal{J}}$, $D= \mathrm{diag}(\varepsilon_{\nu}^{1/2})_{\nu\in\mathcal{J}_{\mathrm{ext}}}$ and $\tilde{D}= \mathrm{diag}(D,I)$. 
	The operator inequality $\Vert TS\phi_0\Vert^2_{H^1}\leq\Vert S \Vert^2_{\mathcal{B}(H^1)} \Vert T\phi_0\Vert^2_{H^1} $ yields with Eq.~\eqref{eq:estimationSchneider4.15} that $\Vert t\Vert^2_{\mathcal{V}}\sim \Vert DA\tilde{D}^{-1}\tilde{D}s\Vert^2_{2}$. 
	We estimate $\Vert DA\tilde{D}^{-1}\Vert_2$ by means of Lemma \ref{lemma:FockEstimate66}:
	\begin{itemize}
		\item[i)] Let $\mu\in \mathcal{J}_{\mathrm{ext}}$. Then
		\begin{equation*}
		\tilde{a}_{\nu,\mu} 
		= 
		\left(\frac{\varepsilon_{\nu}}{\varepsilon_{\mu}}\right)^{1/2}\sum_{\alpha\in\mathcal{J}_{\mathrm{ext}}}
		t_{\alpha}\langle \phi_{\nu}, X_{\alpha} \phi_{\mu}  \rangle 
		\lesssim 
		\sum_{\alpha\in\mathcal{J}_{\mathrm{ext}}}
		t_{\alpha}\varepsilon_{\alpha}^{1/2}\langle \phi_{\nu}, X_{\alpha} \phi_{\mu} \rangle ~.
		\end{equation*}
		\item[ii)] Let $\mu\in \mathcal{J}_{\mathrm{CAS}}$. Then
		\begin{equation*}
		\tilde{a}_{\nu,\mu} 
		= 
		\varepsilon_{\nu}^{1/2}\sum_{\alpha\in\mathcal{J}_{\mathrm{ext}}}
		t_{\alpha}\langle \phi_{\nu}, X_{\alpha} \phi_{\mu}  \rangle 
		\lesssim 
		\sum_{\alpha\in\mathcal{J}_{\mathrm{ext}}}
		t_{\alpha}\varepsilon_{\alpha}^{1/2}\langle \phi_{\nu}, X_{\alpha} \phi_{\mu} \rangle ~.
		\end{equation*}
	\end{itemize}
	Hence, $
	\Vert DA\tilde{D}^{-1} \Vert_{2}^2 \leq C  \sum_{\alpha\in \mathcal{J}_{\mathrm{ext}}}t_{\alpha}^2\varepsilon_{\alpha} 
	=
	C\Vert t\Vert_{\mathcal{V}_{\mathrm{ext}}}^2$ and $\Vert T \Vert_{\mathcal{B}\left(H^1\right)}\leq C\Vert t\Vert^2_{\mathcal{V}_{\mathrm{ext}}}$. The norm equivalence follows since $\Vert t\Vert_{\mathcal{V}_{\mathrm{ext}}} \sim \Vert T\psi\Vert_{H^1} \sim \Vert T\Vert_{\mathcal{B}\left( H^1\right)}$.
\end{proof}

We show the applicability of Lemma~\ref{Th:Zarantonello} by establishing Lipschitz continuity of the TCC function.
\begin{theorem}
The function $f:\mathcal{V}_{\mathrm{ext}}\rightarrow \mathcal{V}_{\mathrm{ext}}'$, given in Definition~\ref{def:Tcc_function}, is differentiable at $t\in \mathcal{V}_{\mathrm{ext}}$. Furthermore, the derivative is Lipschitz continuous as well as all higher derivatives. In particular, for any ball $B_r(t_*)\subseteq \mathcal{V}_{\mathrm{ext}}$ there exists a Lipschitz constant $L$ depending on $r$ and $t_*$ such that
\begin{equation}
\Vert f(t_1)-f(t_2)\Vert_{\mathcal{V}_{\mathrm{ext}}'}\leq L\Vert t_1-t_2\Vert_{\mathcal{V}_{\mathrm{ext}}}\quad
\end{equation}	%
for $t_1,~t_2\in B_{r}(t_*)$.
\label{Th:LipschitzCont}
\end{theorem}
\begin{proof}
For the derivative of $f$ we find
\begin{equation*}
Df(t):\mathcal{V}_{ext} \to \mathcal{V}_{ext}'
~;s \mapsto 
\langle  
\phi_\mu ,e^{-T}[e^{-T^{\rm CAS}} He^{T^{\rm CAS}} ,S] e^T\phi_0
\rangle~.
\end{equation*}
Note that Theorem~\ref{th:OperatorNormEquiv} yields $T^\dag \in \mathcal{B}(H^{-1})$ for any cluster amplitude vector $t \in \mathcal{V}_\mathrm{ext}$.
Then, using $H : H^1 \to H^{-1}$ we obtain $|\langle Df(t)s,u\rangle| \leq C \Vert s \Vert_{\mathcal{V}_\mathrm{ext}} \Vert u \Vert_{\mathcal{V}_\mathrm{ext}}$ for given $s,u\in \mathcal{V}_\mathrm{ext}$.
This shows the boundedness of $f'(t): \mathcal{V}_\mathrm{ext} \to \mathcal{V}_\mathrm{ext}'$, hence, $f:\mathcal{V}_{\mathrm{ext}}\rightarrow \mathcal{V}_{\mathrm{ext}}'$ is differentiable at $t\in \mathcal{V}_{\mathrm{ext}}$.
The continuity of the Coulomb potential \cite{yserentant} and the fluctuation potential $W=H-F$ \cite{lieb1977hartree} further implies the continuity of $t\mapsto f'(t)$.	
Hence $f$ is local Lipschitz continuous on $B_r(t_*)$. 
Higher order derivatives are treated in the same way. 
\end{proof}

To prove that $f$ is locally strongly monotone,  we use the decomposition 
\begin{align} \label{eq:H-split}
H= F + PWP + (W - PWP)~,
\end{align}
where $W$ is the fluctuation operator and $P$ is the orthogonal projection onto the CAS. The decomposition is motivated from a perturbation theory point of view as follows: Suppose $
\lambda=\Vert W - PWP\Vert_{\mathcal{B}(H^1,H^{-1})} = 0$. Then it is straightforward to see that $\mathcal{H}_{\text{CAS}}$ is an invariant subspace for $H$, and hence the CAS FCI problem is exact. Therefore, $t_*=0$ is a solution in this case, as can easily be checked. Also, the CAS-ext gap at least intuitively indicates that the TCC function $f$ is locally strongly monotone at $t_*=0$. (This can also be checked.) Now, suppose $\lambda = \|W - PWP\|_{\mathcal{B}(H^1,H^{-1})}$ is finite and sufficiently small. It is reasonable to expect that $t_*(\lambda)$ is correspondingly small, i.e., a small perturbation of the case $W - PWP = 0$, staying within the domain of strong monotonicity. In conclusion, we expect that under some smallness assumption on $W - PWP$ it is achievable to demonstrate local strong monotonicity of the TCC function $f$. We also note that by enlarging the CAS, $W - PWP$ becomes smaller, so that tuning the CAS can be an important tool to achieve proper smallness in practice.

For a fixed $T^\mathrm{CAS}$, we define the map 
\begin{align*}
O:\mathcal{V}_{\rm ext} \to H^{-1}~;~
t \mapsto \big(e^{-T} (W_{\rm CAS} - PW_{\rm CAS}P)  e^T - (W_{\rm CAS}-PW_{\rm CAS}P)  \big) \phi_0~,
\end{align*}
where $W_{\rm CAS} = \exp(-T^{\rm CAS})W\exp(T^{\rm CAS})$.
Similarly to Theorem~\ref{Th:LipschitzCont}, we find that $O(\cdot)$ is differentiable with  
\begin{equation*}
DO(s): \mathcal{V}_{ext} \to H^{-1}; \quad t\mapsto [e^{-S} (W_\mathrm{CAS}-PW_\mathrm{CAS} P)  e^S
,T] \phi_0~,
\end{equation*}
which implies locally Lipschitz continuity.
For technical reasons, we will make use of a Lipschitz condition with respect to the $l^2$-norm, which is no restriction since all norms are equivalent in finite dimensions.

\textbf{Assumption} (B). \textit{
There exists a ball $B_{\delta}(t_*) \subset \mathcal V_\mathrm{ext}$ such that 
for $t_1,t_2\in B_{\delta}(t_*)$ we have
\[
\Vert O(t_1) -O(t_2)  \Vert_{L^2} \leq L_* \Vert t_1-t_2\Vert_{2} ~,
\]
where the Lipschitz constant $L_*>0$ fulfills
\begin{equation}
\varepsilon_0 - \omega_0   - \Omega_\mathrm{CAS} > L_*   ~,
\label{eq:LischitzConst}
\end{equation}
with $ \Omega_\mathrm{CAS} =  \sum_{ \sigma\in\mathcal{J}_{\rm CAS}}   |t_{\sigma}^\mathrm{CAS}\varepsilon_{\sigma}|$, $\omega_0 = \langle \phi_0, W_\mathrm{CAS} \phi_0 \rangle$ and $\varepsilon_0$ the previously defined CAS-ext gap.
} 

\begin{rmk}
We note that the assumption of Lipschitz continuity of $O(\cdot)$ in Assumption (B) is more than actually needed. 
The crucial requirement is 
\begin{equation*}
|\langle
(T_1-T_2)\phi_0 , O(t_1) - O(t_2)
\rangle| 
\leq C_* \Vert t_1-t_2 \Vert_{2}^2~,
\end{equation*}
for some relatively small $C_*$.
However, this constant $C_*$ can be bounded from above in terms of the Lipschitz constant $L_*>0$ of $O(\cdot)$ since $C_* \leq C L_*$, where by Proposition~\ref{prop:Focknorm_ExtNorm} a constant $C$ exists fulfilling $\Vert t\Vert_{\mathcal{V}_{\mathrm{ext}}}\leq C \Vert  T \phi_0\Vert_{H^1}$. 
Furthermore,  
\begin{align}
\Vert DO(s) \Vert_{\mathcal{B}(L^2)} \sim \delta_{W_\mathrm{CAS}} &:= \Vert W_\mathrm{CAS}-PW_\mathrm{CAS}P \Vert_{\mathcal{B}(L^2)} \nonumber \\
&\leq \sum_k \frac{1}{k!}\Vert[ W-PWP, T^\mathrm{CAS}]_{(k)} \Vert_{\mathcal{B}(L^2)}~, 
\label{eq:WrelTCAS}
\end{align}
such that $L_*\sim \delta_{W_\mathrm{CAS}}$ and $C_*$ fulfills Eq.~~\eqref{eq:LischitzConst} under the assumption that $W-PWP$ is sufficiently small related to $T^\mathrm{CAS}$ as displayed in the rhs. of Eq.~\eqref{eq:WrelTCAS}.  
The latter aligns with a perturbational viewpoint of the TCC method as outlined above.
Note that we \emph{do not} impose a norm restriction on $W$ itself but an ideal CAS, meaning that the multireference character is captured within the CAS, i.e., $PWP$. 
The norm restriction on $W-PWP$ then becomes a natural consequence of the optimal CAS choice. 
\end{rmk}

\begin{rmk}
Note that since $\phi_\mathrm{CAS} = e^{T^\mathrm{CAS}} \phi_0$ is an approximate solution on the CAS, $~\omega_0$ accounts for the non-trivial energy correction (vis-a-vis $\phi_0$) and thus is negative for quantum-molecular systems. 
Typically then, $\omega_0 < 0$ and the CAS-ext gap $\varepsilon_0$ together with $|\omega_0|$ have to be large enough such that $\varepsilon_0 + |\omega_0|   >  \Omega_\mathrm{CAS}$.
Furthermore, Assumption (B) allows $t_\sigma^\mathrm{CAS}$ to be relatively large for $\sigma \in \mathcal J_\mathrm{CAS}$ with $\varepsilon_\sigma$ small. 
A not too big $\Omega_\mathrm{CAS}$ can be guaranteed if $\{ \lambda_j\}_{j=1}^k$ is densely confined because $\vert \varepsilon_\sigma \vert \leq N(\lambda_k - \lambda_1)$.
\end{rmk}

We are now able to prove that $f$ is locally strongly monotone.
\begin{theorem}
Under Assumption (A) and (B), the TCC function $f$ is locally strongly monotone on $B_\delta(t_*)$ for some $\delta>0$.
\label{Th:StronglyMon}
\end{theorem}
\begin{proof}
Let $t_1,t_2\in B_\delta(t_*)\subseteq\mathcal{V}_{\mathrm{ext}}$ and write the Hamiltonian as in Eq.~\eqref{eq:H-split}.
With the notation $\delta_f = \langle f(t_1)-f(t_2),t_1-t_2\rangle$, $\delta_T = T_1-T_2$ and $H_{t_i}=e^{-T_i}He^{T_i}$, the definition of the TCC function $f$ and Lemma \ref{Lemma:FockOperator} yield 
\begin{equation*}
\begin{aligned}
\delta_f &=
\langle \delta_T \phi_0, e^{-T^\mathrm{CAS}} (H_{t_1} -H_{t_2}) e^{T^\mathrm{CAS}} \phi_0 \rangle\\
&=
\langle \delta_T \phi_0, e^{-T^\mathrm{CAS}} [F,\delta_T] e^{T^\mathrm{CAS}} \phi_0 \rangle
+\langle \delta T \phi_0,( e^{-T_1} PW_{\rm CAS}P -e^{-T_2} PW_{\rm CAS}P ) \phi_0\rangle \\
&\quad+ 
\langle \delta_T \phi_0, O(t_1)-O(t_2) \rangle\\
&= \delta_1 + \delta_2+ \delta_3~,
\end{aligned}
\end{equation*}
where the last equality defines $\delta_1,\delta_2$ and $\delta_3$. \newline
\indent To bound $\delta_1$ from below, we first note that Lemma \ref{Lemma:FockOperator} implies
\begin{equation*}
[F, e^{ T^\mathrm{CAS}} ] = \sum_{n=1}^{N} \frac 1 {n!} \sum_{\mu \in \mathcal J_\mathrm{CAS}} (t_{\mathrm{CAS}}^{(n)})_\mu X_\mu
= S~.
\end{equation*}
Since $S$ commutes with $e^{\pm T^\mathrm{CAS}}$ and $\delta_T$, we obtain
\begin{equation*}
\begin{aligned}
e^{-T^\mathrm{CAS}}  [ F,\delta_T] e^{T^\mathrm{CAS}} 
&=
e^{-T^\mathrm{CAS}} ( (S+e^{T^{\mathrm{CAS}}}F) \delta_T -\delta_T (S+e^{T^{\mathrm{CAS}}}F) ) = 
F \delta_T - \delta_T  F~,
\label{eq:comm}
\end{aligned}
\end{equation*}
and consequently $\delta_1 = \langle \delta_T \phi_0, (F - \Lambda_0)  \delta_T\phi_0 \rangle = \sum_{\mu \in \mathcal J_\mathrm{ext}} \varepsilon_\mu (t_1-t_2)_\mu^2$. \newline
\indent Next we find 
\begin{equation}
\label{eq:delta_2_reform}
\begin{aligned}
\delta_2&=\langle \delta T \phi_0,( e^{-T_1} PW_{\rm CAS} -e^{-T_2} PW_{\rm CAS} ) \phi_0\rangle\\
&=
-\langle \delta T \phi_0, \delta T  PW_{\rm CAS} \phi_0\rangle
+
\sum_{k=2}^{\infty}\frac{(-1)^k}{k!}
\langle \delta T \phi_0,( T_2^k-T_1^k) PW_{\rm CAS} \phi_0\rangle\\
&=
-\sum_{\mu\in\mathcal{J}_{\rm ext}}(t_1-t_2)_{\mu}^2\langle  \phi_0, PW_{\rm CAS} \phi_0\rangle\\
&\quad-\sum_{\substack{\mu\neq\nu\in\mathcal{J}_{\rm ext}\\\mu\ominus\nu\in {\rm CAS}}}
(t_1-t_2)_{\mu}(t_1-t_2)_{\nu}\langle  \phi_{\mu\ominus \nu}, PW_{\rm CAS} \phi_0\rangle\\
&\quad+
\sum_{k=2}^{\infty}\frac{(-1)^k}{k!}
\langle \delta T \phi_0,( T_2^k-T_1^k) PW_{\rm CAS} \phi_0\rangle~.\\
\end{aligned}
\end{equation}
We now define $\delta\Psi = \phi_{\rm CAS} -  \psi_{\rm CAS}^{\rm (FCI)}$ with $\phi_{\rm CAS} = \exp(T^{\rm CAS})\phi_0 \approx \psi_{\rm CAS}^{\rm (FCI)}$, where $PHP\Psi_{\rm CAS}^* = E_{\rm CAS}^{\rm (FCI)}\Psi_{\rm CAS}^*$. We know that 
\begin{equation*}
\begin{aligned}
PW_{\rm CAS}P\phi_0
&=
Pe^{-T^{\rm CAS}}H\phi_{\rm CAS}
-Pe^{-T^{\rm CAS}}Fe^{T^{\rm CAS}}\phi_{0}\\
&=
E_{\rm CAS}^{\rm (FCI)}Pe^{-T^{\rm CAS}}\Psi_{\rm CAS}^*
+ Pe^{-T^{\rm CAS}}H\delta\Psi
-P(F+[F,T^{\rm CAS}])\phi_{0}\\
&=
E_{\rm CAS}^{\rm (FCI)}\phi_{ 0}
+ Pe^{-T^{\rm CAS}}(H - E_{\rm CAS}^{\rm (FCI)} )\delta\Psi
-P(F+[F,T^{\rm CAS}])\phi_{0}~.\\
\end{aligned}
\end{equation*}
Since we are merely interested in the projections onto $\phi_{\sigma}$ with  $\sigma \in \mathcal{J}_{\rm CAS}$, we set $\mathcal{R}= \langle \phi_{\sigma},Pe^{-T^{\rm CAS}}(H - E_{\rm CAS}^{\rm (FCI)} )\delta\Psi\rangle$
and obtain
\begin{equation}
\label{eq:ProjectionExplicit}
\begin{aligned}
\langle
\phi_{\sigma}, PW_{\rm CAS} \phi_0
\rangle
&=
\langle \phi_{\sigma},Pe^{-T^{\rm CAS}}(H - E_{\rm CAS}^{\rm (FCI)} )\delta\Psi\rangle
-\langle \phi_{\sigma},[F,T^{\rm CAS}]\phi_{0}\rangle\\
&=
\mathcal{R}
+\langle
\phi_{\sigma}, T^{\rm CAS}F\phi_0
\rangle
-\langle
\phi_{\sigma}, FT^{\rm CAS}\phi_0
\rangle\\
&=
\mathcal{R}
+\sum_{\mu}t_{\mu}^\mathrm{CAS}\Lambda_0
\langle
\phi_{\sigma}, \phi_\mu
\rangle
-\sum_{\mu}t_{\mu}(\Lambda_0+\varepsilon_\sigma)\langle
\phi_{\sigma}, \phi_\mu
\rangle\\
&=
\mathcal{R}
+t_{\sigma}\Lambda_0
-t_{\sigma}(\Lambda_0+\varepsilon_\sigma)
=
\mathcal{R}
-t_\sigma\varepsilon_\sigma ~.
\end{aligned}
\end{equation}
The quantity $\mathcal{R} \sim \Vert \delta \Psi\Vert_{L^2}$ is directly steerable by the used CAS method.
Hence, assuming $\phi_{\rm CAS}\approx \psi_{\rm CAS}^{\rm (FCI)}$ to be a sufficiently good approximation eliminates the above $\mathcal{R}$ dependence.
Inserting Eq.~\eqref{eq:ProjectionExplicit} in the second term of Eq.~\eqref{eq:delta_2_reform}, we find
\begin{equation}
\label{eq:Rough_estimate_for_garbage_term}
\begin{aligned}
&\sum_{\substack{\mu\neq\nu\in\mathcal{J}_{\rm ext}\\\mu\ominus\nu\in {\rm CAS}}}
|(t_1-t_2)_{\mu}(t_1-t_2)_{\nu}|t_{\mu\ominus \nu}|\varepsilon_{\mu\ominus \nu}|\\
&\leq
\Big(
\sum_{\substack{\mu\neq\nu\in\mathcal{J}_{\rm ext}\\\mu\ominus\nu\in {\rm CAS}}}
(t_1-t_2)_{\mu}^2|t_{\mu\ominus \nu}\varepsilon_{\mu\ominus \nu}|
\Big)^{\frac{1}{2}}
\times
\Big(
\sum_{\substack{\mu\neq\nu\in\mathcal{J}_{\rm ext}\\\mu\ominus\nu\in {\rm CAS}}}
(t_1-t_2)_{\nu}^2|t_{\mu\ominus \nu}\varepsilon_{\mu\ominus \nu}|
\Big)^{\frac{1}{2}}\\
&\leq
\Omega_{\rm CAS}
\Vert t_1-t_2\Vert_2^2~,
\end{aligned}
\end{equation}
where we recall that $\Omega_{\rm CAS} = \sum_{\sigma\in\mathcal{J}_{\rm CAS}}|t_{\sigma}\varepsilon_{\sigma}|$ as defined in Assumption (B).
Since $\omega_0 = \langle\phi_0,W_{\rm CAS}\phi_0 \rangle$ and $\Vert T_1^k-T_2^k\Vert_{L^2}\in\mathcal{O}(\Vert t_1-t_2\Vert_{2}^k)$, we conclude with Proposition~\ref{prop:Focknorm_ExtNorm} that
\begin{equation}
\begin{aligned}
\label{eq:delta_2_new}
\delta_2
&\geq
- (\omega_0  + \Omega_\mathrm{CAS})\Vert t_1-t_2\Vert^2_2
+\mathcal{O}(\Vert t_1-t_2\Vert_{\mathcal V_\mathrm{ext}}^3)~.
\end{aligned}
\end{equation}

For the last term, Assumption (B) implies that
\begin{align*}
    \delta_3 \geq - \Vert \delta_T \phi_0 \Vert_{L^2} \Vert O(t_1)-O(t_2) \Vert_{L^2} \geq -L_* \Vert t_1-t_2\Vert_2^2 ~.
\end{align*}
Combining the different bounds above and assuming that $\varepsilon_0$, $\omega_0$, $\Omega_\mathrm{CAS}$, and $L_*$ fulfill Eq.~\eqref{eq:LischitzConst}, we conclude the existence of a $\lambda\in (0,1)$ such that 
\begin{align*}
\delta_f &\geq
\lambda \sum_\mu \varepsilon_\mu (t_1-t_2)_\mu^2 + \sum_\mu \Big[  (1-\lambda)\varepsilon_\mu  - \omega_0   - \Big(L_*+ \Omega_\mathrm{CAS} \Big)\Big](t_1-t_2)_\mu^2 \\ 
& \quad +\mathcal{O}(\Vert t_1-t_2\Vert_{\mathcal V_\mathrm{ext}}^3) \\
&\geq \lambda \Vert t_1-t_2 \Vert_{\mathcal{V}_{\mathrm{ext}}}^2 + \sum_\mu \Big[  (1-\lambda)\varepsilon_0  - \omega_0   - \Big(L_*+ \Omega_\mathrm{CAS} \Big)\Big](t_1-t_2)_\mu^2 \\
& \quad +\mathcal{O}(\Vert t_1-t_2\Vert_{\mathcal V_\mathrm{ext}}^3) \\
&\geq \lambda \Vert t_1 - t_2 \Vert_{\mathcal V_\mathrm{ext}}^2   +\mathcal{O}(\Vert t_1-t_2\Vert_{\mathcal V_\mathrm{ext}}^3) \geq \gamma \Vert t_1 - t_2 \Vert_{\mathcal V_\mathrm{ext}}^2 \sim  \Vert \delta_T \phi_0 \Vert_{H^1}^2   ~.
\end{align*}
In the last step we have assumed $\delta$ to be sufficiently small such that $\mathcal{O}(\Vert t_1-t_2\Vert_{\mathcal V_\mathrm{ext}}^3)$ can be absorbed. 
\end{proof}

\begin{rmk}
We note that Eq.~\eqref{eq:Rough_estimate_for_garbage_term} is a pessimistic estimation, since we neglect the conditions of the excitation indices, i.e., $\mu\neq\nu$ such that $\mu \ominus \nu \in\mathcal{J}_{\rm CAS}$. 
This restriction means that the excitation rank of the CC method dictates which CAS amplitudes are considered.
In particular, considering the DMRG-TCCSD method we find
\begin{equation*}
\begin{aligned}
\sum_{\substack{
\nu\in\mathcal{J}_{\rm ext}\\
\nu\neq\mu\\
\mu\ominus\nu\in {\rm CAS}}
}
|t_{\mu\ominus \nu}\varepsilon_{\mu\ominus \nu}|
\leq
\sum_{\sigma\in \mathcal{J}_{\rm CAS}^{(1)}}
|t_{\sigma}\varepsilon_{\sigma}|~,
\end{aligned}
\end{equation*}
where the superscripted $\mathcal{J}_{\rm CAS}^{(1)}$ means that only single-excitations on the CAS are considered---which correspond to orbital rotations.
Moreover, without loss of generality one can assume Brueckner type orbitals, implying that this term vanishes.
\end{rmk}

By Theorem \ref{Th:LipschitzCont} and \ref{Th:StronglyMon}, we can apply Lemma~\ref{Th:Zarantonello} to the TCC function $f$ ensuring a locally unique and quasi-optimal approximate solutions.
Next, we will show quadratic convergence of tailored coupled-cluster methods which aligns the non-variational TCC approach with any variational method in terms of convergence speed.  

\subsection{Error Estimate}
\label{Sec:Error}
In this section we present an estimate for the energy error introduced by truncating the TCC method, e.q, DMRG-TCCSD.
In comparison to the single-reference CC method, the error is divided into different parts as a consequence of the basis splitting. 
The TCC function is typically parameterized by an approximation $T^{\mathrm{CAS}}$ of the FCI solution $T_{\mathrm{FCI}}^{\mathrm{CAS}}$ on $\mathcal{H}_{\mathrm{CAS}}$. 
We emphasize that $T_{\mathrm{FCI}}^{\mathrm{CAS}}$ is in itself an approximation of the inaccessible $T_*^{\mathrm{CAS}}$ (cf. Theorem~\ref{Thm:Implication} and the following discussion). 
This of course influences the error and is here accounted for. 
On top of that, the truncation error of the CC method applied to $\phi_{\mathrm{CAS}}=e^{T^{\mathrm{CAS}}} \phi_0$ enters.     
For this part of the error we follow the analysis of the  single-reference CC methods and use the Aubin-Nitsche-duality method for nonlinear Galerkin schemes, see \cite{rohwedder2013error}.
We consider $d$-dimensional approximation spaces $\mathcal V_{\mathrm{ext}}^{(d)}$, $d \leq \vert \mathcal J \vert $, of the external amplitude space $\mathcal V_{\mathrm{ext}}$. 
For a given $T^{\mathrm{CAS}}$ we denote $t_d \in \mathcal V_{\mathrm{ext}}^{(d)}$ the solution of $P_{d}f(\,\cdot\,;t^{\mathrm{CAS}})|_{\mathcal{V}_{\mathrm{ext}}^{(d)}}=0$, where $P_{d}$ is the $l^2$-orthogonal projection onto $(\mathcal{V}_{\mathrm{ext}}^{(d)})'$.
Thus, $t_d$ is an approximation of the full solution $t_*\in\mathcal V_{\mathrm{ext}}$, where $t_*$ solves $f(\,\cdot\,; t^{\mathrm{CAS}})=0$ on $\mathcal V_{\mathrm{ext}}$.

\begin{rmk}
    In practice, the space $\mathcal{V}_\text{ext}^{(d)}$ is constructed by restricting the CC amplitudes to a particular subspace, e.g., allowing excitations from the reference $\phi_0$ into the external space of rank less than a fixed number, say, including up to singles and doubles.
    This choice is practical (the dimension $d$ is fairly low), however, the alternative truncation that allows excitations from \emph{any} CAS determinant $\phi_\alpha$ into the external space of rank less than a fixed number yields what is called the first-order interaction space~\cite{lyakh2012multireference}. While the dimension can be much higher than the previous choice, it gives external correlation energies guaranteed to be correct through second order in $H_1 = W - PWP$. In other words, all CAS determinants are treated on equal footing, which is essential for an optimal multireference treatment. The first truncation scheme puts special significance to the reference $\phi_0$. 
    \end{rmk}
    
\indent
We will here derive a general error estimate valid for every choice of method used on $\mathcal{H}_{\mathrm{CAS}}$ potentially introducing an additional error on the CAS denoted $\delta E_\mathrm{CAS}$.
In notational consistency with the introduction of Section~\ref{Sec:Analysis}, let $\psi_* = e^{T_*^{\mathrm{ext}}} e^{T_*^{\mathrm{CAS}}} \phi_0 $ be the exponential parameterization of the FCI solution on $\mathcal{H}_K$. 
Then, the energy error is subsequently split as follows
\begin{equation}
\label{eq:EestMain}
\begin{aligned}
\delta E &=
| \mathcal E(t_d;t^{\mathrm{CAS}}) -\mathcal E(t_*^{\mathrm{ext}};t_*^{\mathrm{CAS}})| \\
&\leq
| \mathcal E(t_d;t^{\mathrm{CAS}}) - \mathcal{E}(t_*;t^{\mathrm{CAS}}  )| +| \mathcal E(t_*;t^{\mathrm{CAS}}) - \mathcal{E}(t_*;t_{\mathrm{FCI}}^{\mathrm{CAS}}  )| \\
& \quad+| \mathcal E(t_*;t_{\mathrm{FCI}}^{\mathrm{CAS}}) - \mathcal{E}(t_*^{\mathrm{ext}};t_*^{\mathrm{CAS}}  )|  \\
& =: \delta \varepsilon + \delta \varepsilon_{\mathrm{CAS}} + \delta \varepsilon_{\mathrm{CAS}}^*~, 
\end{aligned}
\end{equation}
where the last equality defines the different error terms.

The quantity $\delta\varepsilon$ describes the error produced by truncating the TCC method parameterized by $\phi_{\mathrm{CAS}}=e^{T^{\mathrm{CAS}}}\phi_0$.
The second term $\delta\varepsilon_{\mathrm{CAS}}$ is connected to the usage of an approximate solution $\psi_{\mathrm{CAS}}=e^{T^{\mathrm{CAS}}}\phi_0$ on $\mathcal{H}_{\mathrm{CAS}}$ instead of the FCI solution $\phi^{(\mathrm{FCI})}_{\mathrm{CAS}}=e^{T_{\mathrm{FCI}}^{\mathrm{CAS}}}\phi_0$.
We introduce $\tilde t_*\in \mathcal V_{\mathrm{ext}}$ that solves $f(\tilde t_*; t_{\mathrm{FCI}}^{\mathrm{CAS}})=0$. Note that the pair $(\tilde t_*,t_{\mathrm{FCI}}^{\mathrm{CAS}}) \in \mathcal V_{\mathrm{CAS}} \times \mathcal V_{\mathrm{ext}}$ is the best solution possible using a given basis splitting. We emphasize, in comparison, that $t_* = (t_*^{\mathrm{CAS}},t_*^{\mathrm{ext}})$ is a theoretical construct where the basis splitting has been done after computing $t_*$.

The main result of this section is given below in Theorem~\ref{thm:err_main}. The idea is to bound $\delta E$ by means of the splitting above. We introduce the error $\delta E_{\mathrm{CAS}}$ in the following way:
The wavefunction $e^{T_{\mathrm{FCI}}^{\mathrm{CAS}}}\phi_0$ is in general not an eigenfunction of $H$, however, it is an eigenfunction of $PHP$ where $P$ is the orthogonal projection on $\mathcal H_{\mathrm{CAS}}$. 
We then define
\begin{equation}
\label{eq:Ecas}
\delta E_{\mathrm{CAS}} =|\langle
\phi_0, \big(e^{-T^{\mathrm{CAS}}}PHPe^{T^{\mathrm{CAS}}}-e^{-T_{\mathrm{FCI}}^{\mathrm{CAS}}}PHPe^{T_{\mathrm{FCI}}^{\mathrm{CAS}}}\big)\phi_0
\rangle|~.
\end{equation}
The energy difference $\delta E_{\mathrm{CAS}}$ describes the error induced by an approximation to the FCI solution on $\mathcal{H}_{\mathrm{CAS}}$.
We emphasize that this error depends on the approximation method used.
Using the DMRG, which is a variational method, yields a quadratic
error bound.

The error $\delta \varepsilon$ is estimated using similar techniques as described in Ref.~\cite{rohwedder2013error}.
To that end, we define the following Euler-Lagrange systems.
For notational simplicity we drop again the explicit parameterization by $t^{\mathrm{CAS}}$.
We consider the functionals 
\[
\langle f(t),\cdot \rangle: \mathcal{V}_{\mathrm{ext}} \to \mathbb{R};~ u\mapsto \langle  U\phi_{0}, e^{-T^{\mathrm{CAS}}}e^{-T}He^{T}e^{T^{\mathrm{CAS}}}\phi_0  \rangle
\]
and
\[
\mathcal{E}(\cdot): \mathcal{V}_{\mathrm{ext}} \to \mathbb{R};~ u\mapsto \langle  \phi_{0}, e^{-T^{\mathrm{CAS}}}e^{-U}He^{U}e^{T^{\mathrm{CAS}}}\phi_0  \rangle~.
\]
We note that $\langle f(t),\cdot \rangle$ is a real-valued linear form whereas $\mathcal{E}(\cdot)$ is a nonlinear functional.
The corresponding variational problem 
\begin{equation}
\langle f(t),u \rangle=0\quad, \forall u \in \mathcal{V}_{\mathrm{ext}}
\label{eq:Clustereq}
\end{equation}
describes the cluster equations.
The associated Galerkin approximation on $\mathcal{V}_{\mathrm{ext}}^{(d)}\subseteq \mathcal{V}_{\mathrm{ext}}$ determines $t_d\in \mathcal{V}_{\mathrm{ext}}^{(d)}$ such that
\begin{equation}
\langle f(t_d),u_d \rangle=0\quad, \forall u_d \in \mathcal{V}_{\mathrm{ext}}^{(d)}~.
\label{eq:ClustereqDis}
\end{equation}

We use the Euler-Lagrange method to estimate the error $\mathcal{E}(t)-\mathcal{E}(t_d)$. 
Introducing the dual variable $z\in \mathcal{V}_{\mathrm{ext}}$, we define the Lagrangian
\begin{equation}
\mathcal{L}:\mathcal{V}_{\mathrm{ext}}\times \mathcal{V}_{\mathrm{ext}}\to \mathbb{R};\quad
(t,z)\mapsto \mathcal{E}(t)-\langle f(t),z \rangle,
\label{eq:Langrangian}
\end{equation}
and seek for stationary points $(t_*,z_*)\in\mathcal{V}_{\mathrm{ext}}\times \mathcal{V}_{\mathrm{ext}}$ of $\mathcal{L}(\cdot, \cdot)$, i.e., 
\begin{equation}
\mathcal{L}'(t_*,z_*)(u,v)
=\left\lbrace
\begin{aligned}
\mathcal{E}'(&t_*)u - \langle f'(t_*)u,z_* \rangle\\
&- \langle f(t_*),v \rangle
\end{aligned}
\right\rbrace=0 ~,
\label{eq:StatEL}
\end{equation}
for all $(u,v)\in\mathcal{V}_{\mathrm{ext}}\times \mathcal{V}_{\mathrm{ext}}$.
The Galerkin approximations $(t_d,z_d)\in\mathcal{V}_{\mathrm{ext}}^{(d)}\times \mathcal{V}_{\mathrm{ext}}^{(d)}$ are defined by the discrete Euler-Lagrange system
\begin{equation}
\mathcal{L}'(t_d,z_d)(u_d,v_d)
=\left\lbrace
\begin{aligned}
\mathcal{E}'(&t_d)u_d - \langle f'(t_d)u_d,z_d \rangle\\
&- \langle f(t_d),v_d \rangle
\end{aligned}
\right\rbrace=0 ~,
\label{eq:StatELDis}
\end{equation}
for all $(u_d,v_d)\in\mathcal{V}_{\mathrm{ext}}^{(d)}\times \mathcal{V}_{\mathrm{ext}}^{(d)}$.
We remark that in both situations \eqref{eq:StatEL} and \eqref{eq:StatELDis}, the $t$- respectively the $t_d$-component of any stationary point is a solution of the cluster equations and the discrete cluster equations, respectively.

The main results of this section now reads:
\begin{theorem}
	\label{thm:err_main}
	Let $\mathscr{B}=\{\chi_1,...,\chi_K\}\subseteq H^1$ be a set of $L^2$-orthonormal spin-orbitals that are split into $\mathscr{B}_{\mathrm{CAS}}$ and $\mathscr{B}_{\mathrm{ext}}$. 
	We denote $\mathcal{H}_K$ and $\mathcal{H}_{\mathrm{CAS}}$ the FCI space corresponding to $\mathscr{B}$ resp.~$\mathscr{B}_{\mathrm{CAS}}$. 
	Let further $t_*^{\mathrm{CAS}}\in \mathcal{V}_{\mathrm{CAS}}$ be the projection of the FCI amplitudes on $\mathcal{H}_K$ onto $\mathcal{H}_{\mathrm{CAS}}$, $t_{\mathrm{FCI}}^{\mathrm{CAS}}\in \mathcal{V}_{\mathrm{CAS}}$ the FCI amplitudes on $\mathcal{H}_{\mathrm{CAS}}$, and $t^{\mathrm{CAS}}\in \mathcal{V}_{\mathrm{CAS}}$ an approximation to $t_{\mathrm{FCI}}^{\mathrm{CAS}}$. Let $\mathcal{V}_{\mathrm{ext}}^{(d)}\subset \mathcal{V}_{\mathrm{ext}}$ be a subspace fulfilling 
	\begin{equation}
	d(t_*,\mathcal{V}_{\mathrm{ext}}^{(d)})\leq \frac{ \gamma \delta}{\gamma+L}~,
	\label{eq:SuffCond}
	\end{equation}
	where $\gamma,L>0$ are the monotonicity and Lipschitz
        constants of $f(\,\cdot\,;t^{\mathrm{CAS}})$ on
        $B_\delta(t_*)$. Then there is a unique solution $t_d \in
        \mathcal V_{\mathrm{ext}}^{(d)}$ of $P_d f(\,\cdot\,;t^{\mathrm{CAS}})|_{\mathcal{V}_{\mathrm{ext}}^{(d)}}=0$ that approximates the solution $t_*\in\mathcal V_{\mathrm{ext}}$ of $f(\,\cdot\,; t^{\mathrm{CAS}})=0$ on $\mathcal V_{\mathrm{ext}}$. 
        Let $(z_d,z_*)\in\mathcal{V}_{\mathrm{ext}}^{(d)}\times\mathcal{V}_{\mathrm{ext}}$ be the corresponding dual solutions of $(t_d,t_*)\in\mathcal{V}_{\mathrm{ext}}^{(d)}\times\mathcal{V}_{\mathrm{ext}}$. 
        Further, set $\tilde{t}_*\in \mathcal V_{\mathrm{ext}}$ the solution of $f(\,\cdot\,; t^{\mathrm{CAS}}_{\mathrm{FCI}})=0$ on $\mathcal V_{\mathrm{ext}}$ and $t_*^{\mathrm{ext}}\in\mathcal V_{\mathrm{ext}}$ the projection of the FCI amplitudes on $\mathcal{H}_K$ onto $\mathcal{H}_{\mathrm{CAS}}^{\perp}$.
	It then follows that the energy error can be bounded as	
	\begin{align*}
	\delta E &\lesssim  \Vert t_d-t_*\Vert_{\mathcal{V}_{\mathrm{ext}}}\left(\Vert t_d-t_*\Vert_{\mathcal{V}_{\mathrm{ext}}} +\Vert z_d-z_*\Vert_{\mathcal{V}_{\mathrm{ext}}}\right) + 	\Vert t_* - t_*^{\mathrm{ext}}\Vert_{\mathcal{V}_{\mathrm{ext}}}^2  + \Vert t_* - \tilde t_*\Vert_{\mathcal V_{\mathrm{ext}}}^2\\
	& \quad    + \Vert t_{\mathrm{FCI}}^{\mathrm{CAS}}-t_*^{\mathrm{CAS}}  \Vert_{2}^2     + \Vert t^{\mathrm{CAS}} - t_{\mathrm{FCI}}^{\mathrm{CAS}}\Vert_{2}^2
	+ \sum_{\substack{
		\mu\in\mathcal{J}_{\mathrm{ext}}\\|\mu|=1}
		} \varepsilon_\mu (\tilde t_*)_\mu^2 ~ + \delta E_{\mathrm{CAS}}~.
	\end{align*}
\end{theorem}
\begin{rmk}
The energy error estimate in Theorem~\ref{thm:err_main} holds for any basis splitting fulfilling the presented conditions.
However, in the extremal cases of a minimal or maximal basis splitting, i.e., $k=N$ and $k=K$, the TCC method collapses to the CC and CAS method,respectively.
\end{rmk}
\begin{rmk}
	Since we do not have an equivalence of Theorem~\ref{th:OperatorNormEquiv} for sequences over $\mathcal J_{\mathrm{CAS}}$ ($\varepsilon_\mu$ are not guaranteed to be strictly greater than zero for $\mu\in\mathcal J_{\mathrm{CAS}}$), we instead bound the sequences over $\mathcal J_{\mathrm{CAS}}$ using the unweighted $l^2$-norm.
\end{rmk}

We will prove Theorem~\ref{thm:err_main} by first establishing a series of lemmas that relates to the r.h.s. of Eq.~\eqref{eq:EestMain}. We start with the term $
\delta \varepsilon^*_{\mathrm{CAS}} =  | \mathcal E(t_*;t_{\mathrm{FCI}}^{\mathrm{CAS}}) - \mathcal{E}(t_*^{\mathrm{ext}};t_*^{\mathrm{CAS}}  )|$.
\begin{lemma}
	\label{lemma:deltaEcasStar}
	Under the assumptions of Theorem \ref{thm:err_main} the following bound holds
	\begin{align*}
	\delta \varepsilon_{\mathrm{CAS}}^*
	\lesssim
	\Vert  t_* - t_*^{\mathrm{ext}}\Vert_{\mathcal{V}_{\mathrm{ext}}}^2+ \Vert t_{\mathrm{FCI}}^{\mathrm{CAS}}-t_*^{\mathrm{CAS}}  \Vert_{2}^2~.
	\end{align*}
\end{lemma}
\begin{proof}
	Recall that $\psi_*=e^{T_*^{\mathrm{ext}}} e^{T_*^{\mathrm{CAS}}} \phi_0$ corresponds to the FCI solution on $\mathcal H_K$ and consequently $D\mathcal E(t_*^{\mathrm{ext}};t_*^{\mathrm{CAS}})=0 $. Taylor expanding $\mathcal E( t_*;t_{\mathrm{FCI}}^{\mathrm{CAS}})$ around $(t_*^{\mathrm{CAS}} ,t_*^{\mathrm{ext}})$ yields
	\begin{align*}
	\mathcal E(t_*;t_{\mathrm{FCI}}^{\mathrm{CAS}}) - \mathcal{E}(t_*^{\mathrm{ext}};t_*^{\mathrm{CAS}}  )  
	=  \frac{1}{2}   D^2  \mathcal E(t_*^{\mathrm{ext}};t_*^{\mathrm{CAS}})((e,\tilde{e}),(e,\tilde{e})) + \mathcal{R}^{(3)} ~,
	\end{align*}
	where $\tilde{e} =  t_* - t_*^{\mathrm{ext}}$, $e = t_{\mathrm{FCI}}^{\mathrm{CAS}}-t_*^{\mathrm{CAS}}$ and $\mathcal{R}^{(3)}$ describes the third order error term.
	For $H_{t_1+t_2}= e^{-T_1} e^{-T_2} H e^{T_2} e^{T_1}$ with amplitudes $t_1\in \mathcal{V}_{\mathrm{ext}}$ and $t_2\in \mathcal{V}_{\mathrm{CAS}}$ we compute 
	\begin{align*}
	(D^2\mathcal{E}(t_1;t_2))_{\mu,\nu} 
	= 
	\langle\phi_0, [[H_{t_1+t_2},X_{\nu}],X_{\mu}]\phi_0\rangle
	=
	\langle\phi_0, H_{t_1+t_2}X_{\nu}X_{\mu}\phi_0\rangle~.
	\end{align*}
	Thus, with $H_* = H_{t_*^{\mathrm{ext}} + t_*^{\mathrm{CAS}}}$ and 
	\[\delta_{ \tilde T} = \sum_{\mu \in \mathcal J_{\mathrm{ext}}}  ( t_* - t_*^{\mathrm{ext}})_\mu X_\mu~, \qquad  
	\delta_T = \sum_{\mu \in \mathcal J_{\mathrm{CAS}}}  ( t_{\mathrm{FCI}}^{\mathrm{CAS}}-t_*^{\mathrm{CAS}})_\mu X_\mu \]
	we have 
	\begin{align*}
	D^2  \mathcal E(t_*^{\mathrm{ext}};t_*^{\mathrm{CAS}})((e,\tilde{e}),(e,\tilde{e})) 
	&= 
	\langle \phi_0, H_* (\delta_{ \tilde T} + \delta_T)^2 \phi_0\rangle \\ 
	&\leq 2 \langle \phi_0, H_* \delta_{ \tilde T}^2 \phi_0\rangle + 2\langle \phi_0, H_* (\delta_T)^2 \phi_0\rangle~.
	\label{Taylor3}
	\end{align*}
	Using Theorem \ref{th:OperatorNormEquiv}, as well as the boundedness of $H$, we obtain
	\begin{align*}
	\langle \phi_0, H_* \delta_{ \tilde{T}}^2 \phi_0\rangle 
	\leq
	C \Vert \phi_0 \Vert_{H^1}^2 \Vert \delta_{\tilde{T}} \Vert_{\mathcal{B}(H^1)}^2
	\leq 
	C \Vert  t_* - t_*^{\rm ext} \Vert_{\mathcal{V}_{\rm ext} }^2~.
	\end{align*}
	By direct computation, we bound the term $\langle \phi_0, H_* (\delta_T)^2 \phi_0\rangle$ using the $l^2(\mathcal J_{\mathrm{CAS}})$ norm
	\begin{align*}
	\langle \phi_0, H_* (\delta_T)^2 \phi_0\rangle
	&\leq C \Vert \delta_T \Vert_{\mathcal B(H^1)}^2
	=C \Vert \sum_{\mu \in \mathcal J_{\mathrm{CAS}}}  ( t_{\mathrm{FCI}}^{\mathrm{CAS}}-t_*^{\mathrm{CAS}})_\mu X_\mu \Vert_{\mathcal B(H^1)}^2 \\
	&\leq C  \sum_{\mu \in \mathcal J_{\mathrm{CAS}}}  ( t_{\mathrm{FCI}}^{\mathrm{CAS}}-t_*^{\mathrm{CAS}})_\mu^2 \Vert X_\mu \Vert_{\mathcal B(H^1)}^2 
	\leq C \Vert t_{\mathrm{FCI}}^{\mathrm{CAS}}-t_*^{\mathrm{CAS}} \Vert_{2}^2~.
	\end{align*}
\end{proof}

Next, we analyze the energy difference $\delta \varepsilon_{\mathrm{CAS}} = | \mathcal E(t_*;t^{\mathrm{CAS}}) - \mathcal{E}(t_*;t_{\mathrm{FCI}}^{\mathrm{CAS}}  )|$.
\begin{lemma}
	\label{lemma:deltaEcas}
	Under the assumptions of Theorem \ref{thm:err_main} the following bound holds
	\begin{align*}
	\delta \varepsilon_{\mathrm{CAS}} \lesssim  \delta E_{\mathrm{CAS}} + \Vert t_* - \tilde t_*\Vert_{\mathcal V_{\mathrm{ext}}}^2 + \Vert ( T^{\mathrm{CAS}} - T_{\mathrm{FCI}}^{\mathrm{CAS}})\phi_0\Vert_{H^1}^2 
	+ \sum_{|\mu|=1} \varepsilon_\mu (\tilde t_*)_\mu^2~.
	\end{align*}
\end{lemma}
\begin{proof}
Starting from the definition of $\delta \varepsilon_{\mathrm{CAS}}$, we obtain straightforwardly 
\begin{align*}
\delta \varepsilon_{\mathrm{CAS}}
\leq
|\langle
\phi_0, \big(e^{-T^{\mathrm{CAS}}}He^{T^{\mathrm{CAS}}}-e^{-T_{\mathrm{FCI}}^{\mathrm{CAS}}}He^{T_{\mathrm{FCI}}^{\mathrm{CAS}}}\big)\phi_0
\rangle|+\mathcal{R}~,
\end{align*}
where 
$
\mathcal{R} =|\langle\phi_0, \big[\big(e^{-T^{\mathrm{CAS}}}He^{T^{\mathrm{CAS}}}-e^{-T_{\mathrm{FCI}}^{\mathrm{CAS}}}He^{T_{\mathrm{FCI}}^{\mathrm{CAS}}}\big),e^{T_*}\big]\phi_0\rangle|$.
Since $\phi_0,\,e^{T_{\mathrm{FCI}}^{\mathrm{CAS}}}\phi_0$ and $e^{T^{\mathrm{CAS}}}\phi_0$ are elements of $\mathcal{H}_\mathrm{CAS}$, we find
\begin{align*}
\delta \varepsilon_{\mathrm{CAS}}-\mathcal{R}
&\leq
|\langle
\phi_0, \big(e^{-T^{\mathrm{CAS}}}He^{T^{\mathrm{CAS}}}-e^{-T_{\mathrm{FCI}}^{\mathrm{CAS}}}He^{T_{\mathrm{FCI}}^{\mathrm{CAS}}}\big)\phi_0
\rangle|\\
&\leq
|\langle
\phi_0, \big(e^{-T^{\mathrm{CAS}}}PHPe^{T^{\mathrm{CAS}}}-e^{-T_{\mathrm{FCI}}^{\mathrm{CAS}}}PHPe^{T_{\mathrm{FCI}}^{\mathrm{CAS}}}\big)\phi_0
\rangle|\\
&\quad+
|\langle
\phi_0, \big(\big[T^{\mathrm{CAS}},P\big]HPe^{T^{\mathrm{CAS}}}-\big[T_{\mathrm{FCI}}^{\mathrm{CAS}},P\big]HPe^{T_{\mathrm{FCI}}^{\mathrm{CAS}}}\big)\phi_0
\rangle|~.
\end{align*}
For any excitation operator $X=\sum_{\mu\in\mathcal{J}_\mathrm{CAS}}c_{\mu}X_{\mu}$, we remark that $XP\psi\in \mathcal{H}_\mathrm{CAS}$ for all $\psi\in \mathcal{H}_K$.
By definition of $\mathcal{H}_\mathrm{CAS}$ we also find $XQ\psi\in \mathcal{H}_\mathrm{ext}$ for all $\psi\in \mathcal{H}_K$, where $Q=I-P$.
Therefore $X=(P+Q)X(P+Q)=PXP+QXQ$ and consequently $[X,P]=[PXP,P]=0$.
Hence, $\big[T^{\mathrm{CAS}},P\big]=\big[T_{\mathrm{FCI}}^{\mathrm{CAS}},P\big]=0$.
In particular,
\begin{align*}
\delta \varepsilon_{\mathrm{CAS}}
&\leq
|\langle
\phi_0, \big(e^{-T^{\mathrm{CAS}}}PHPe^{T^{\mathrm{CAS}}}-e^{-T_{\mathrm{FCI}}^{\mathrm{CAS}}}PHPe^{T_{\mathrm{FCI}}^{\mathrm{CAS}}}\big)\phi_0
\rangle|
+
\mathcal{R}=
\delta E_{\mathrm{CAS}}+\mathcal{R}~,
\end{align*}
where $\delta E_{\mathrm{CAS}}$ is defined by Eq.~\eqref{eq:Ecas}. 
To estimate $\mathcal R$ we consider the splitting of the Hamilton operator $H=F+W$.
Note that $[T^{\mathrm{CAS}},T_*]=[T_{\mathrm{FCI}}^{\mathrm{CAS}},T_*]=0$ which implies together with Lemma \ref{Lemma:FockOperator} that the $F$-dependent terms in $\mathcal R$ vanish. 
The Baker--Campbell--Hausdorff expansion and the fact that $((T_*)^m)^{\dagger}\phi_0 = 0$ for all $m\geq 1$ then yields
\begin{align*}
\mathcal{R}
=
|\langle
\phi_0, \Big(\sum_{m=1}\frac{1}{m!}[W,e^{T^{\mathrm{CAS}}}]_m-\sum_{m=1}\frac{1}{m!}[W,e^{T_{\mathrm{FCI}}^{\mathrm{CAS}}}]_m\Big)\sum_{m=1}\frac{1}{m!}(T_*)^m\phi_0
\rangle|~.
\end{align*}
Since $W$ is a two-particle operator, the Slater--Condon rules imply that the non-zero contributions in the above expansion are given for $m=1$ and only by the single-excitation parts of the respective operators.
It then follows with $((T^{\mathrm{CAS}})_1)^{\dagger}\phi_0 =((T_{\mathrm{FCI}}^{\mathrm{CAS}})_1)^{\dagger}\phi_0 = 0$ that
\begin{align*}
\mathcal R 
&= |\langle \phi_0, W (T^{\mathrm{CAS}} - T_{\mathrm{FCI}}^{\mathrm{CAS}})_1(T_*)_1\phi_0\rangle|~,
\end{align*}
where $(\cdot)_1$ denotes the single-excitation part of the respective operator. We then estimate  
\begin{align*}
\mathcal R 
&\leq |\langle \phi_0,	W (T^{\mathrm{CAS}} - T_{\mathrm{FCI}}^{\mathrm{CAS}}  )_1 (T_* - \tilde T_*)_1  \phi_0 \rangle| +|\langle \phi_0,	W (T^{\mathrm{CAS}} - T_{\mathrm{FCI}}^{\mathrm{CAS}}  )_1 (\tilde T_* )_1  \phi_0 \rangle| \\
&\leq \left(C_1 \Vert T_* - \tilde T_*\Vert_{\mathcal{B}(H^1)}
+ C_2\Vert (\tilde T_* )_1\Vert_{\mathcal{B}(H^1)}\right) \Vert (T^{\mathrm{CAS}} - T_{\mathrm{FCI}}^{\mathrm{CAS}}  )\phi_0\Vert_{H^1} \\
&\leq \frac{C_1} 2 \Vert T_* - \tilde T_*\Vert_{\mathcal{B}(H^1)}^2  + \frac{C_2} 2 \Vert (\tilde{T}_*)_1\Vert_{\mathcal{B}(H^1)}^2  + \frac{C_1+C_2} 2   \Vert (T^{\mathrm{CAS}} - T_{\mathrm{FCI}}^{\mathrm{CAS}}  )\phi_0\Vert_{H^1}^2~ .
\end{align*}
Hence, $
\mathcal R \leq D_1 \Vert t_* - \tilde t_*\Vert_{\mathcal V_{\mathrm{ext}}}^2 +D_2 \Vert ( T^{\mathrm{CAS}} - T_{\mathrm{FCI}}^{\mathrm{CAS}})\phi_0\Vert_{H^1}^2 
+ D_3\sum_{|\mu|=1} \varepsilon_\mu (\tilde t_*)_\mu^2$.
\end{proof}

For the remaining error $\delta \varepsilon$ we use techniques that have been developed by Bangerth and Rannacher for a general functional analytic framework \cite{bangerth2013adaptive}.  
Hence, under the assumption that $f$ is locally strongly monotone the following analysis holds also in the $K\to \infty$ limit.
Nevertheless, before passing on to the error estimate of $\delta \varepsilon$ we characterize the approximation space $\mathcal{V}_{\mathrm{ext}}^{(d)}$.
Let $\{b_1,\dots,b_D \}$ be a basis of $\mathcal{V}_{\mathrm{ext}}$, and without loss of generality, $\{b_1,\dots,b_d \}$ be the corresponding subbasis of $\mathcal{V}_{\mathrm{ext}}^{(d)}$ with $d<D$.
A key aspect for the analysis is $\mathcal{V}_{\mathrm{ext}}^{(d)}$ being a sufficiently good approximation of $\mathcal{V}_{\mathrm{ext}}$.
Subsequently, we elaborate a sufficient condition for this to hold.
Let $\delta>0$ be chosen according to Assumption~(B) such that Theorem \ref{Th:LipschitzCont} and \ref{Th:StronglyMon} imply $f$ being strongly monotone and Lipschitz continuous on $B_\delta(t_*)$ with constants $\gamma$ and $L$. 
Further, we define 
\[
\kappa_d = d(t_*,\mathcal{V}_{\mathrm{ext}}^{(d)})= \min_{t_d\in \mathcal{V}_{\mathrm{ext}}^{(d)}} \Vert t_d-t_*\Vert_{\mathcal{V}_{\mathrm{ext}}}.
\]
Eq.~\eqref{eq:SuffCond} in Theorem~\ref{thm:err_main} yields the assumption $\kappa_d \leq \gamma \delta/(\gamma+L)$. 
Then, the truncated cluster equation $f|_{\mathcal{V}_{\mathrm{ext}}^{(d)}}=0$ has a locally unique solution on $\mathcal V_{\mathrm{ext}}^{(d)} \cap B_\delta(t_*)$. 
We adapt the proof of Theorem 4.1 in \cite{rohwedder2013error}, which rests on the following consequence of Brouwer's fixed point theorem \cite{emmrich2013gewohnliche}:

\begin{theorem}[Brouwer, 1965]
\label{th:Brouwer}
    Equip $\mathbb R^d$ with any norm $\|\cdot\|_d$.
	Let $B_R$ be the closed ball of radius $R$ centered at $x=0$ and $h : B_{R} \to \mathbb R^d$ be continuous.
	If $\langle h( x), x\rangle \geq 0$ on $\partial B_R$ then $h( x)=0$ for some $ x\in B_R$.
\end{theorem}

Let $t_{\mathrm{opt}}\in \mathcal{V}_{\mathrm{ext}}^{(d)}$ with $\kappa_d= \Vert t_{\mathrm{opt}}-t_*\Vert_{\mathcal{V}_{\mathrm{ext}}}$, we define the continuous function $h_d:\mathbb R^{d}\to \mathbb R^{d};\,x\mapsto (y_j)_{j=1}^d$, where $y_j = \langle f(t_{\mathrm{opt}} + v),b_j\rangle $ and $v=\sum_{j=1}^d x_j b_j$.
We chose $\Vert x \Vert_{d} = \Vert v \Vert_{\mathcal V_{\mathrm{ext}}^{(d)} }$ as a norm on $\mathbb R^{d}$.
Then, $h_d(t)=0$ if and only if $f(t)|_{\mathcal V_{\mathrm{ext}}^{(d)} } =0$.
By assumption $ \delta -\kappa_d \geq \delta L/(\gamma +L)>0$ and we set $R= \delta-\kappa_d$. Then $v \in B_R(t_{\mathrm{opt}})$ implies $v\in B_\delta(t_*)$. 
Assuming further $\Vert  x\Vert_{d}=R $, the monotonicity and Lipschitz continuity of $ f$ then yield
\begin{align*}
\langle h_d( x), x\rangle
&=
\sum_{j=1}^{d}\langle f(t_{\mathrm{opt}} + v),b_j\rangle x_j
=
\langle f(t_{\mathrm{opt}} + v)-f(t_{\mathrm{opt}}),v\rangle+\langle f(t_{\mathrm{opt}})-f(t_*),v\rangle \\
&\geq \gamma \Vert v \Vert^2_{\mathcal V_{\mathrm{ext}}^{(d)}}+L\kappa_d\Vert v \Vert_{\mathcal V_{\mathrm{ext}}^{(d)}}
=
R(\gamma R+L\kappa_d)~.
\end{align*}
Since $\gamma R - L\kappa_d = \gamma  \delta - \kappa_d (\gamma +L)\geq 0$, we conclude $\langle h_d( x), x\rangle = R(\gamma R - L\kappa_d)\geq 0$.
By Theorem~\ref{th:Brouwer} this yields  $h_d( x_*)=0$ for some $ x_*$ with $\Vert  x_* \Vert_{d}  \leq  R $, which is equivalent to $t_d = t_{\mathrm{opt}} + v_*$ solving the projected problem $f|_{\mathcal V_{\mathrm{ext}}^{(d)} }=0$. 
The uniqueness follows from Theorem \ref{Th:Zarantonello} applied to $f|_{\mathcal V_{\mathrm{ext}}^{(d)} }$.

In the sequel, we assume that  $\mathcal{V}_{\mathrm{ext}}^{(d)}$ is a sufficiently good approximation of $\mathcal{V}_{\mathrm{ext}}$ as guaranteed by Eq.~\eqref{eq:SuffCond}.
We note that the Lagrangian \eqref{eq:Langrangian} is nonsymmetric, consequently we cannot expect the error to be quadratic with respect to the error of the wavefunction.
However, we see that the dual variable $z$ enters in \eqref{eq:Langrangian}. 
Indeed, in the analysis that will follow, the solution $z_*$ of the dual problem enters the error estimates. 
In the spirit of \cite{rohwedder2013error}, we start the estimation of $\delta \varepsilon$ with a lemma that concerns the dual solution.
\begin{lemma}
\label{Lemma:QuasiOptimalDualSol}
Let $f$ be strongly monotone on $B_{\delta}(t_*)$, then there exists a unique dual solution $z_*\in\mathcal{V}_{\mathrm{ext}}$ determined by $t_*$ such that $(t_*,z_*)$ is a stationary point of the Lagrangian $\mathcal{L}(\cdot,\cdot)$, i.e., $(t_*,z_*)$ solves \eqref{eq:StatEL}.
Additionally, there exists a corresponding unique $z_d\in \mathcal{V}_{\mathrm{ext}}^{(d)}$ such that $(t_d,z_d)$ solves the discretized equation \eqref{eq:StatELDis} and approximates the exact dual solution quasi-optimally in the sense that 
\begin{equation}
\Vert z_d-z_*\Vert_{\mathcal{V}_{\mathrm{ext}}} \leq c_1 \Theta_d +c_2\Theta_d^2~,
\label{eq:QuasiOptDualElement}
\end{equation}
with $\Theta_d=\max\big\lbrace d(t_*,\mathcal{V}_{\mathrm{ext}}^{(d)})~,~d(z_*,\mathcal{V}_{\mathrm{ext}}^{(d)})\big\rbrace$.
\end{lemma}
\begin{proof}
	By definition $t_*$ solves the second component of \eqref{eq:StatEL}. 
	Therefore it remains to show the fist equation.
	To that end we use Lax--Milgram \cite{evans2010partial}, for which we need to establish boundedness and coercivity of $f'(t_*)^{\dag}$.
	First, we note that the boundedness of $f'(t_*)$ was shown in Theorem \ref{Th:LipschitzCont}.
	Secondly, we expand $f$ into a Taylor series at $t_*$, i.e., $f(t_*+w)-f(t_*)=f'(t_*)w+\mathcal{O}(\Vert w\Vert^2_{\mathcal{V}_{\mathrm{ext}}})$ with $w\in B_{\delta}(t_*)$.
	The strong monotonicity estimate then yields $\langle f'(t_*)w,w\rangle \geq \gamma \Vert w\Vert^2_{\mathcal{V}_{\mathrm{ext}}}-\mathcal{O}(\Vert w\Vert^3_{\mathcal{V}_{\mathrm{ext}}})$.
	For an arbitrary $u$ we choose $c\in \mathbb{R}$ sufficiently large such that $w=u/c\in B_{\delta}(t_*)$. This implies the coercivity of $f'(t_*)$.
	Thirdly, we remark that boundedness and coercivity of $f'(t_*)$ are transferred straightforwardly to the adjoint operator $f'(t_*)^{\dag}$.
	We set $a(z_*,u)=\langle f'(t_*)^{\dag} z_* ,u\rangle$ and apply Lax--Milgram to the equation $a(z_*,u)=\mathcal{E}'(t_*)(u)$ for all $u\in \mathcal{V}_{\mathrm{ext}}$.
	This yields the existence and uniqueness of $z_* \in \mathcal{V}_{\mathrm{ext}}$.

	This argumentation holds whenever $f$ is strongly monotone.
	Hence, the existence and uniqueness of $z_d$ follows by the assumption that $\mathcal{V}_{\mathrm{ext}}^{(d)}$ is a sufficiently good approximation to $\mathcal{V}_{\mathrm{ext}}$.
	To show Eq.~\eqref{eq:QuasiOptDualElement} we decompose $z_d-z_*=z_d-\tilde{z}_d+\tilde{z}_d-z_*$, where $\tilde{z}_d\in \mathcal{V}_{\mathrm{ext}}^{(d)}$ solves 
	\begin{equation}
	(\mathcal{E}'(t_*))(u_d) = \langle f'(t_*)u_d,\tilde{z}_d\rangle~,\quad\forall u_d\in \mathcal{V}_{\mathrm{ext}}^{(d)}~.
	\label{eq:LemmaCeaTCCProb}
	\end{equation}
	Note that this is not the discrete problem since it uses the solution $t_*$ instead of $t_d$.
	In the same manner as we previously defined $a(\cdot,\cdot)$ we define a bilinear form from \eqref{eq:LemmaCeaTCCProb}.
	Because $f'(t_*)$ is a bounded and coercive linear map, C\'ea's lemma \cite{zaidler1990nonlinear} implies the quasi optimal approximation by $\tilde{z}_d$ to $z_*$, i.e., $\Vert\tilde{z}_d-z_* \Vert_{\mathcal{V}_{\mathrm{ext}}}\leq C~ d(z_*,\mathcal{V}_{\mathrm{ext}}^{(d)})$.

	To estimate $\Vert z_d-\tilde{z}_d \Vert_{\mathcal{V}_{\mathrm{ext}}}$ we use the coercivity of $f'(t_d)$.
	From Eqs.~\eqref{eq:LemmaCeaTCCProb} and \eqref{eq:StatELDis} we deduce
	\begin{align*}
	\gamma &\Vert z_d-\tilde{z}_d \Vert^2_{\mathcal{V}_{\mathrm{ext}}}
	\leq
	\langle f'(t_d)(z_d-\tilde{z}_d),z_d-\tilde{z}_d \rangle\\
	&=
	(\mathcal{E}'(t_d)-\mathcal{E}'(t_*))(z_d-\tilde{z}_d)+\langle (f'(t_*)-f'(t_d))(z_d-\tilde{z}_d),\tilde{z}_d\rangle\\
	&\leq
	L_{\mathcal{E}'}\Vert t_d-t_* \Vert_{\mathcal{V}_{\mathrm{ext}}}\Vert z_d-\tilde{z}_d \Vert_{\mathcal{V}_{\mathrm{ext}}}+L_{f'}\Vert t_d-t_* \Vert_{\mathcal{V}_{\mathrm{ext}}}\Vert z_d-\tilde{z}_d \Vert_{\mathcal{V}_{\mathrm{ext}}}\Vert \tilde{z}_d \Vert_{\mathcal{V}_{\mathrm{ext}}}\\
	&=
	(L_{\mathcal{E}'}+L_{f'}\Vert \tilde{z}_d \Vert_{\mathcal{V}_{\mathrm{ext}}})\Vert t_d-t_* \Vert_{\mathcal{V}_{\mathrm{ext}}}\Vert z_d-\tilde{z}_d \Vert_{\mathcal{V}_{\mathrm{ext}}}~.
	\end{align*}
	Using the quasi optimality of $\Vert\tilde{z}_d-z_* \Vert_{\mathcal{V}_{\mathrm{ext}}}$ we find that $\Vert\tilde{z}_d\Vert_{\mathcal{V}_{\mathrm{ext}}}$ is bounded by $\Vert z_* \Vert_{\mathcal{V}_{\mathrm{ext}}}+C d(z_*,\mathcal{V}_{\mathrm{ext}}^{(d)})$ and therefore
	\begin{align*}
	\Vert z_d-\tilde{z}_d \Vert_{\mathcal{V}_{\mathrm{ext}}}
	&\leq
	\frac{1}{\gamma}\left[
	L_{\mathcal{E}'}+L_{f'} \big(\Vert z_* \Vert_{\mathcal{V}_{\mathrm{ext}}}+C~ d(z_*,\mathcal{V}_{\mathrm{ext}}^{(d)})\big)\right]
	\Vert t_d-t_* \Vert_{\mathcal{V}_{\mathrm{ext}}}\\
	&\lesssim
	c_1\,d(t_*,\mathcal{V}_{\mathrm{ext}}^{(d)})
	+
	c_2\,d(t_*,\mathcal{V}_{\mathrm{ext}}^{(d)})~d(z_*,\mathcal{V}_{\mathrm{ext}}^{(d)})~.
	\end{align*}
\end{proof}

In order to estimate the error $\delta\varepsilon=|\mathcal{E}(t_*)-\mathcal{E}(t_d)|$ we define the primal residual $\rho(t_d)(\cdot):\mathcal{V}_{\mathrm{ext}}^{(d)}\to \mathbb{R};\;u\mapsto -\langle f(t_d),u \rangle$ and the dual residual $\rho^*(t_d,z_d)(\cdot):\mathcal{V}_{\mathrm{ext}}^{(d)}\to \mathbb{R};$ $u\mapsto \mathcal{E}'(t_d)(u)-\langle Df(t_d)(u),z_d \rangle$.
The following error characterization is based on the results of Bangerth and Rannacher \cite{bangerth2013adaptive} formulated in a suitable way for this article.
\begin{theorem}[Bangerth--Rannacher, 2003]
\label{th:Bangerth-Rannacher}
    For any solution of Eqs.~\eqref{eq:Clustereq} and \eqref{eq:ClustereqDis}, we have the error representation
	\begin{align}
	2(\mathcal E(t_*)-\mathcal E(t_d))
	=
	\mathcal{R}_d^{(3)}+\rho(t_d)(z_*-\upsilon_d)+\rho^*(t_d,z_d)(t_*-w_d)~,
	\label{eq:Bangerth,Rannacher}
	\end{align}
	with arbitrary $\upsilon_d,w_d\in\mathcal{V}_{\mathrm{ext}}^{(d)}$.
	The remainder term $\mathcal{R}_d^{(3)}$ is cubic in the primal and dual error $e=t_*-t_d$ and $e^*=z_*-z_d$,
	\begin{equation*}
	\begin{aligned}
	\mathcal{R}_d^{(3)}
	&=
	\int_0^1\left(
	\mathcal{E}^{(3)}(t_d+se)(e,e,e)
	-\langle f^{(3)}(t_d+se)(e,e,e),z_d+se^* \rangle\right.\\
	&\qquad\left.
	-3\langle f^{(2)}(t_d+se)(e,e),e^* \rangle\right)s(s-1)~ds~.
	\end{aligned}
	\end{equation*}
\end{theorem}
\indent
Similarly to the approach in \cite{rohwedder2013error} we are able to conclude with the following error estimates for the TCC energy.  
\begin{theorem} \label{Thm:QE}
	Let $\mathcal{V}_{\mathrm{ext}}^{(d)}$ be a sufficiently large subspace of $\mathcal{V}_{\mathrm{ext}}$ in the sense that $\Theta_d<c$ (see Lemma \ref{Lemma:QuasiOptimalDualSol}) for a suitable $c\in(0,1)$, and denote by $(t_*,z_*)$ and $(t_d,z_d)$ the solutions of Eqs.~\eqref{eq:StatEL} and \eqref{eq:StatELDis}.
	If $f$ is strongly monotone at $t_*$, we have
	\begin{align} 
	 \delta\varepsilon
	&\leq 
	\Vert t_d-t_*\Vert_{\mathcal{V}_{\mathrm{ext}}}\left(c_1\Vert t_d-t_*\Vert_{\mathcal{V}_{\mathrm{ext}}} +c_2\Vert z_d-z_*\Vert_{\mathcal{V}_{\mathrm{ext}}}\right)~,
	\label{1}
	\end{align}
	and further
	\begin{subequations}
	\begin{align}
	\label{2.1}
	\delta\varepsilon
	&\lesssim \Big(d(t_*,\mathcal{V}_{\mathrm{ext}}^{(d)}) +d(z_*,\mathcal{V}_{\mathrm{ext}}^{(d)})\Big)^2~,\\
	\label{2.2}
	\delta\varepsilon
	&\lesssim \Vert (e^{T_d}-e^{T_*})\phi_{\mathrm{CAS}}\Vert_{H^1}(\Vert (e^{T_d}-e^{T_*})\phi_{\mathrm{CAS}}\Vert_{H^1}+\Vert (e^{Z_d}-e^{Z_*})\phi_{\mathrm{CAS}}\Vert_{H^1})~,\\
	\label{2.3}
	\delta\varepsilon
	&\lesssim  \Big(\inf_{\psi\in\mathcal{H}_{\mathrm{ext}}}\Vert \psi-e^{T_*}\phi_{\mathrm{CAS}}\Vert^2_{H^1} +
	\inf_{\psi\in\mathcal{H}_{\mathrm{ext}}}\Vert \psi-e^{Z_*}\phi_{\mathrm{CAS}}\Vert^2_{H^1} \Big)^2~.
	\end{align}
	\end{subequations}
\end{theorem}
\begin{proof}
	Using Eq.~\eqref{eq:StatEL} we can rewrite the dual residual as follows:
	\begin{equation*}
	\begin{aligned}
	&\rho^*(t_d,z_d)(s)=(\mathcal{E}'(t_d))(s)-\langle f'(t_d)(s),z_d \rangle\\
	&\quad=
	(\mathcal{E}'(t_d)-\mathcal{E}'(t_*))(s)+ \langle (f'(t_*)-f'(t_d))(s),z_* \rangle+\langle f'(t_d)(s),z_*-z_d \rangle~,
	\end{aligned}
	\end{equation*}
	for an arbitrary $s\in\mathcal{V}_{\mathrm{ext}}$.
	Using Eq.~\eqref{eq:Bangerth,Rannacher} in Theorem~\ref{th:Bangerth-Rannacher} we obtain
	\begin{align*}
	2\delta\varepsilon
	&\leq
	|\mathcal{R}_d^{(3)}|+|\langle f(t_d)-f(t_*),z_*-\upsilon_d \rangle|
	+|(\mathcal{E}'(t_d)-\mathcal{E}'(t_*))(t_*-w_d)|\\
	&\quad+ |\langle (f'(t_*)-f'(t_d))(t_*-w_d),z_* \rangle|+|\langle f'(t_d)(t_*-w_d),z_*-z_d \rangle|~.
	\end{align*}
	Exploiting the different Lipschitz continuities further implies 
	\begin{equation}
	\begin{aligned}
	2\delta\varepsilon
	&\leq
	|\mathcal{R}_d^{(3)}|+L_{f}\Vert t_d-t_* \Vert_{\mathcal{V}_{\mathrm{ext}}}\Vert z_*-\upsilon_d \Vert_{\mathcal{V}_{\mathrm{ext}}}
	+L_{\mathcal{E}'}\Vert t_d-t_* \Vert_{\mathcal{V}_{\mathrm{ext}}}\Vert t_*-w_d \Vert_{\mathcal{V}_{\mathrm{ext}}}\\
	&\quad+L_{f'}\Vert t_d-t_* \Vert_{\mathcal{V}_{\mathrm{ext}}}\Vert t_*-w_d \Vert_{\mathcal V_{\mathrm{ext}}}\Vert z_* \Vert_{\mathcal{V}_{\mathrm{ext}}}\\
	&\quad+C\Vert t_*-w_d \Vert_{\mathcal{V}_{\mathrm{ext}}}\Vert z_*-z_d \Vert_{\mathcal{V}_{\mathrm{ext}}}~.
	\label{eq:QuatError1}
	\end{aligned}
	\end{equation}
	This yields $2\delta\varepsilon
	\leq
	\Vert t_d-t_* \Vert_{\mathcal{V}_{\mathrm{ext}}}(c_1\Vert t_*-t_d \Vert_{\mathcal{V}_{\mathrm{ext}}}
	+c_2\Vert z_*-z_d \Vert_{\mathcal{V}_{\mathrm{ext}}})+|\mathcal{R}_d^{(3)}|
	$ for $w_d=t_d$ and $\upsilon_d=z_d$.
	By straightforward computations we estimate
	\begin{align*}
	|\mathcal{R}_d^{(3)}|
	&\leq
	L_{\mathcal{E}^{(3)}}
	\Vert t_*-t_d\Vert_{\mathcal{V}_{\mathrm{ext}}}^3
	+\zeta L_{f^{(3)}}\Vert t_*-t_d\Vert_{\mathcal{V}_{\mathrm{ext}}}^3 +3L_{f^{(2)}}\Vert t_*-t_d\Vert_{\mathcal{V}_{\mathrm{ext}}}\Vert z_*-z_d\Vert_{\mathcal{V}_{\mathrm{ext}}}\,,
	\end{align*}
	with $\zeta=\max_{s\in [0,1]}\Vert z_d+se^*\Vert_{\mathcal{V}_{\mathrm{ext}}} $.
	Hence, by Lemma \ref{Lemma:QuasiOptimalDualSol}, $|\mathcal{R}_d^{(3)}|\in\mathcal{O}(\Theta_d^3)$, i.e., we can control the remainder term $\mathcal{R}_d^{(3)}$ by means of $ \Theta_d^3$. 
	Since by assumption $\mathcal{V}_{\mathrm{ext}}^{(d)}$ is a sufficiently large subspace of $\mathcal{V}_{\mathrm{ext}}$ in the sense that $\Theta_d<c$, this shows Eq.~\eqref{1}.

	The bound in Eq.~\eqref{2.1} follows from inserting the optimal approximations $t_{\mathrm{opt}}$, $z_{\mathrm{opt}}\in\mathcal{V}_{\mathrm{ext}}^{(d)}$ in \eqref{eq:QuatError1} and applying Theorem~\ref{Th:Zarantonello}, Lemma~\ref{Lemma:QuasiOptimalDualSol} and the fact that $\Theta_d <1$. Then $\Vert z_d-z_*\Vert_{\mathcal{V}_{\mathrm{ext}}} \lesssim \Theta_d$ as the term in $\mathcal{O}(\Theta_d^2)$ becomes negligible.
	The inequalities \eqref{2.2} and \eqref{2.3} follow from Proposition~\ref{prop:Focknorm_ExtNorm}.
\end{proof}
We remark that this error estimate derivation does not require the uniqueness of the solution. 
In cases with not unique solutions, the {\it a priori} assumption $t_d\to t_*$ makes the result meaningful as then the remainder term can be assumed to be small.

We conclude this section by combining previous results to prove Theorem \ref{thm:err_main}, the main result of Subsection \ref{Sec:Error}.

\begin{proof}[Proof of Theorem~\ref{thm:err_main}]
	From Eq.~\eqref{eq:EestMain}, we recall that $\delta E \leq  \delta \varepsilon + \delta \varepsilon_{\mathrm{CAS}} + \delta \varepsilon_{\mathrm{CAS}}^*$. Then using	
	Lemma~\ref{lemma:deltaEcasStar}, Lemma~\ref{lemma:deltaEcas} and Eq.~\eqref{1} in Theorem~\ref{Thm:QE}, the desired result now follows.
\end{proof}

\section{Concluding Remarks and Outlook}
In this article, we presented a first analysis of the TCC method, proving locally unique and quasi-optimal solutions in Theorems \ref{Th:LipschitzCont} and \ref{Th:StronglyMon}, and a direct error estimate given by Theorem \ref{thm:err_main}.
The conceptional change from the HOMO-LUMO gap to the CAS-ext gap $\varepsilon_0$ is a key aspect of this article. 
The definition of $\varepsilon_0$ is tailored for existence, uniqueness and error estimate results that are widely applicable, in particular also to excited state approximations. 
For merely ground-state studies, better bounds in Section \ref{subsec:LocalUniqueness} are obtainable by changing the considered CAS-ext gap to $\tilde\varepsilon_0$.
The extended CAS-ext gap $\tilde\varepsilon_0$ is larger, and in general increases with the size of the CAS. 
Since the gap assumption enters directly in the norm estimates, a connection between the constants involved in the norm estimates in Section~\ref{subsec:LocalUniqueness} and the size of $\mathscr B_{\mathrm{CAS}}$ seems likely but remains to be proven.
Given the presented analysis, it appears reasonable to assume that the results can be generalized to the continuous formulation of the Schr\"odinger equation, without reference to a finite-dimensional single-particle basis for external space.
This corresponds to $K\to \infty$, in which case many of the concepts used in Section \ref{subsec:LocalUniqueness} may be generalized.
Apparently, the main problem in this generalization is that several properties of the Fock operator do not hold for $K\to \infty$. In particular, its spectrum is not purely discrete.
Even though it may appear that we use the Fock operator and its properties excessively, the presented analysis may also be performed for a different one-particle operator, i.e., not necessarily the Fock operator. 
Section \ref{Sec:Error} is based on achievements for general variational problems, implying the validity of Theorem \ref{thm:err_main} for infinite dimensions.   
The currently most important application of our analysis is the DMRG-TCCSD method \cite{veis2016coupled,doi:10.1021/acs.jpclett.6b02912,veis2018full,antalik2019towards}. In a recent publication, we investigated its numerical performance in light of the results in this article~\cite{faulstich2019numerical}.
Using tensor factorization methods---to obtain a well-chosen basis splitting and an approximation to the FCI solution on $\mathcal{H}_{\mathrm{CAS}}$---simplifies the error estimate in Theorem \ref{thm:err_main} since the methodological error becomes negligible and $\delta E_{\mathrm{CAS}}$ is quadratically bound. 
This yields a Galerkin-typical quadratic error estimate for the DMRG-TCC method.

\section{Acknowledgements}
We would like to thank Rolf Heilemann Myhre, Ji\u{r}\'i Pittner, Mih\'aly Andr\'as Csirik and Christian Schilling for valuable discussions and input to this project. 

\bibliographystyle{siamplain}
\bibliography{lib1}
\end{document}